\documentclass[12pt,reqno]{article}

\usepackage{amsthm,amsfonts,amssymb,amsmath,epsf, verbatim}

\newtheorem{theorem}{Theorem}
\newtheorem{lemma}[theorem]{Lemma}
\newtheorem{corollary}[theorem]{Corollary}
\newtheorem{proposition}[theorem]{Proposition}

\newtheorem{conjecture}[theorem]{Conjecture}
\newtheorem{definition}[theorem]{Definition}

\newcommand{\cpx}[1]{\|#1\|}

\newcommand{\ldima}{\lhd_{\alpha}}
\newcommand{\rdima}{\rhd_{\alpha}}

\begin{document}

\author{Joshua Zelinsky\footnote{The author is grateful for a variety of institutions supporting this research Work on this paper occurred when the author was an Assistant Professor at the University of Maine, a Visiting Professor at Birmingham Southern College, and when the author was a Lecturer at Iowa State University. }}
  \date{}

\title{Upper Bounds on Integer Complexity}

\maketitle
\vspace{-1 cm}

\begin{center}
Hopkins School \\ 
zelinsky@gmail.com 

\end{center}

\begin{abstract}
Define $\cpx{n}$ to be the \emph{complexity} of $n$, which is the smallest
number of $1$s needed to write $n$ using an arbitrary combination of addition and multiplication. John Selfridge showed that $\cpx{n}\geq 3\log_3 n$ for all $n$.  Richard Guy noted the trivial upper bound that $\cpx{n} \leq 3\log_2 n$ for all $n>1$ by writing $n$ in base 2. An upper bound for almost all $n$ was provided by Juan Arias de Reyna and  Jan Van de Lune. This paper provides the first non-trivial upper bound for all $n$.  In particular, for all $n>1$ we have $\cpx{n} \leq A \log n$ where $A = \frac{41}{\log 55296}$. 

\end{abstract}

  \maketitle

Define  \emph{complexity} of $n$, the smallest
number of $1$s needed to write $\cpx{n}$ using an arbitrary combination of addition and multiplication. We'll write this as $\cpx{n}$. For example, the equation $$6=(1+1)(1+1+1)$$ shows that $\cpx{6} \leq 5$ and  in fact $\cpx{6}=5$. We will also refer to this function as the integer complexity of $n$.

Integer complexity was first discussed by Kurt Mahler and  Jan Popken in 1953 \cite{MP}.  It was later popularized by
Richard Guy \cite{Guy}, who includes it as problem F26 in his \emph{Unsolved
Problems in Number Theory} \cite{UPINT}.   Since then,  many other authors have explored the behavior of this function, \cite{Raws} especially Juan Arias de Reyna \cite{Arias} and Harry Altman \cite{Altman}.

John Selfridge showed that $\cpx{n}\geq 3\log_3 n$ for all $n$ and noted that this lower bound was achieved whenever $n$ was a power of $3$.  Altman and the author \cite{Altman} classified numbers whose integer complexity was close to this lower bound.

Guy noted the trivial upper bound that $\cpx{n} \leq 3\log_2 n$ for all $n>1$ by writing $n$ in base 2. Prior to this paper, work has been done on providing an upper bound for all $n$ except for a set of density zero. Let $C_a$ be the infimum  of all $c$ such that, for almost all $n$, we have $\cpx{n} \leq c\log n$. Obviously, any such set of density zero must contain $1$. 

It is a folklore result that from writing $n$ in base $2$ one has $$C_a \leq \frac{5}{2\log2} = 3.6067 \cdots .$$ 

Subsequently, John Isbell \cite{Guy} pointed out that writing $n$ in base 24 yields $$C_a \leq \frac{265}{24\log 24} = 3.474 \cdots $$ It is believed based on numerical evidence that the actual maximum of $\frac{\cpx{n}}{\log n}$ occurs at $n=1439$ where $$\frac{\cpx{1439}}{\log 1439}= \frac{26}{\log 1439} = 3.575 \cdots.$$ Thus,  Isbell's result involves values of $c$ which require a non-trivial exceptional set. Whether Isbell's exceptional set is finite or infinite remains an open problem. 

By writing $n$ in base $2^93^6$ 
Juan Arias de Reyna and Jan Van de Lune\cite{ReynaLune}, proved that $$C_a \leq \frac{15903451}{2^93^6\log(2^93^6)} = 3.320 \cdots.$$ Unfortunately, the computations needed to extend to larger bases are difficult. It is unclear whether using larger bases in this fashion is moving to a specific constant. Even if one assumes that they are moving towards on some specific constant,  it is unclear if this constant is $C_a$. 

Recently, Katherine Cordwell, Alyssa Epstein, Anand Hemmady, Steven J. Miller, Eyvindur Palsson, Aaditya Sharma, Stefan Steinerberger, and Yen Nhi Truong  Vu \cite{Williamsgroup} improved on the bound on $C_a$ and also presented a more efficient algorithm for calculating $\cpx{n}$. By writing $n$ in base $2^{11}3^9$, they showed that $$C_a \leq \frac{2326006662
}{2^{11}3^9\log(2^{11}3^9)}.$$

Define $C_M$ to be the smallest number such that for all $n>1$, we have $\cpx{n} \leq C_M \log x$. From Guy's trivial bound, one immediately has $$C_M \leq \frac{3}{\log 2}= 4.32808 \cdots.$$ Improving this bound has been an open problem until this paper. 

We prove:
\begin{theorem}
\label{upperbounddirectimprovement} For all $n>1$ we have $\cpx{n} \leq A \log n$ where $$A= \frac{41}{\log 55296} = 3.7544   \cdots$$
\end{theorem}   

Note that $ \frac{41}{\log 55296} = 3.7544 \cdots $ whereas the bound from just using base 3 has the constant $\frac{3}{\log 2} = 4.32808$. This is a larger improvement than it initially looks, since as noted earlier we actually cannot do better than $\frac{\cpx{1439}}{\log 1439} = 3.575 \cdots.$ Thus, in a certain sense, our result cuts by more than half the worst possible value for $\sup_{n >1}{\cpx{n}{\log n}}$ given the interval in which it must live. 

Our constant is, of course, much larger than that of de Reyna and Van de Lune which is the price we must pay for having a bound that does not have any exceptional set.

Note that our result also gives us the slightly weaker  but somewhat prettier bound of $\cpx{n} \leq \frac{19}{5}\log n$. These results are substantially weaker than:

\begin{conjecture} For all $n>1$, one has $\cpx{n}/\log n \leq 26/(\log 1439)$, with equality obtained only at $n=1439$. 
\end{conjecture} 

It is likely that $C_s=\limsup \cpx{n}/\log n$ is substantially lower  than even $\frac{26}{1439}$ but how much lower is also an open problem. Guy asked if one always has $\cpx{2^a3^b}=2a+3b$. If this is is true then we would have an automatic lower bound of $C_s \geq \frac{2}{\log 2}$ since one would have $$\frac{\cpx{2^a}}{ \log 2^a} = \frac{2a}{a\log2} = \frac{2}{\log 2}.$$ For some progress towards answering this question, see \cite{Altman}. Note that the obvious conjecture for primes other than $2$ or $3$ is false. For example, we have $$\cpx{5^6} = 29 < 30 =6\cpx{5},$$ and $$\cpx{11^2} = 15 < 16 = 2\cpx{11}.$$ In both these cases, the unexpectedly low complexity is arising from $n-1$ having many small prime factors. 

There is a similar problem in the literature where one defines complexity in a way which allows subtractions in addition to multiplication and addition. In that case, the best known bound is in \cite{subtraction group}, which obtains a bound lower than our constant.

Let us now discuss our method of proof for Theorem \ref{upperbounddirectimprovement}. The classical upper bound works simply by writing $n$ in base 2. One might try writing $n$ in a different base, but doing so does not help by itself. If one writes $n$ in base $3$, one gets an upper bound of $\frac{5}{\log 3}\log n  =(4.551 \cdots \log n)$. If one tries a larger base the situation gets even worse. Perhaps surprisingly, the bound one gets for a given base $b$ is not monotonic in $b$. For example, since one can write $6=(1+1)(1+1+1)$, writing in base 6 gives a slightly better constant than base 5. Base 5 gives a constant of $\frac{9}{\log 5}= 5.592 \cdots$, whereas base 6 gives a constant of  $\frac{10}{\log 6}= 5.581 \cdots$. In general though, every base other than $2$ will give by itself a worse upper bound. Our tactic is to use multiple bases. We will write $n$ in terms of each base and then switch bases at the opportune moment as we are writing $n$ in terms of smaller numbers. We will have multiple options for what order to switch bases, and we will show that at least one order will always work. Our method of proof to make this productive will be essentially inductive. 

\begin{definition} We will write $S(\alpha,n)$ to mean the statement ``If for all $1<k<n$, $\cpx{k} \leq \alpha \log k$ then $\cpx{n} \leq  \alpha \log n$.''
\end{definition}

 The statement $S(\alpha,n)$ amounts to the induction step of a proof that for all $n>1$, $\cpx{n} \leq \alpha \log n$. The standard proof that $\cpx{n} \leq \frac{3}{\log 2} \log n$ by writing $n$ in base $2$ then amounts to proving $S(\frac{3}{\log 2},n)$ for all $n>1$ by checking two cases, $n \equiv 0$ (mod 2) and $n \equiv 1$ (mod 2).   Note that for some values of $\alpha$ and $n$, this statement will be vacuously true if the initial hypothesis is false.

To prove our main theorem, we will assume that $S(\alpha,n)$ fails for $\alpha \geq \frac{41}{\log 55296}$, and by using our base-switching argument derive a set of linear inequalities for $v_p(n+1)$ for various primes $p$, where $v_p(m)$ is the $p$-adic valuation of $m$. One key interpretation is that $v_p(n+1)$ is exactly how many digits of $p-1$ that $n$ ends in when written in base $p$.\\

Before we prove our main theorem, we will prove a slightly weaker version of the main theorem that uses only bases $2$ and base $3$. In particular we prove the following:

\begin{theorem}\label{weakupperbound} If $$\alpha \geq \frac{5}{\log 2 + \frac{1}{2}\log 3}$$ then for all $n>1$ we have $\cpx{n} \leq \alpha \log n$. 
\end{theorem}

The proof of Theorem \ref{weakupperbound} has the advantage of illustrating all the major ideas of the main theorem, but since it has a lot fewer moving parts, it is substantially easier to follow before we tackle the main theorem. 

This paper is divided into seven sections:

The first section establishes basic preliminaries and notation needed. 

The second section proves Theorem \ref{weakupperbound},  a weaker version of our main theorem which uses only bases 2 and 3. This weaker bound  should make the general method clear. In that section we will derive two inequalities under the assumption that $S(\alpha,n)$ fails. One of the inequalities corresponds to first reducing using base 2 and then switching to base 3 at the right time. The other corresponds to first reducing using base 3 and then switching to base 2 at the right time. The first inequality will give a lower bound for $v_2(n+1)$ in terms of $v_3(n+1)$ and the second inequality will give a lower bound for $v_3(n+1)$ in terms of $v_2(n+1)$.  We will show that for a specific choice of $\alpha$ the resulting pair of inequalities do not have a solution, and so we conclude that the induction step always works. This amounts essentially to saying that we can either get a reasonably efficient representation of $n$ by either first writing $n$ in base 2 and then switching to base 3 or doing the reverse. To make this work, we will also need some slight understanding of what happens modulo 5.

The fact that these reductions work can be thought of as essentially observing that $-1/3$ is purely periodic in the 2-adics with part $01$, and that $-1/2$ is purely periodic in the 3-adics with repeating part just $1$. If a reader prefers a more concrete interpretation, one can instead read these as statements about $1/3$ in base 2 and $1/2$ in base 3, since in general $-1/n$ in the $p$-adics has the same expansion as $1/n$ in base $p$. 

Why do these quantities matter?  Let's say that we have a number $n$ which is uncooperative at being reduced using obvious base reductions. The obvious way that a number can be maximally uncooperative for being reduced by a given base $b$ is if all the last few digits are $b-1$ when written in base $b$, so when one is trying to write the number in terms of smaller numbers by repeatedly subtracting and then dividing by $b$, one has to subtract the largest possible amount. To make this situation more concrete, let us fix for now $b=3$. Then this amounts to saying that $n$ when written in base $3$ looks like $(\cdots) \cdots 22222$ where the $(\cdots)$ represents other digits. Now, it may be that n is uncooperative in multiple bases so $n$ in base $2$ looks like $(\cdots) \cdots 111$. Now, we might hope that subtracting 1 and dividing by 2 would give us a number that was nice in base 3, but we are not that lucky. If we take $\frac{n-1}{2}$ and write in base $3$ we will find that it still ends in just as many 2s as $n$ did. This can be thought of as arising from the fact that in the 3-adics we have $\cdots 222 =-1$. 

But if we keep subtracting $1$ and dividing by $2$, eventually we will have an even number (after we do this  $v_2(n+1)$ times) and then we can divide by 2. If we have a number in base $3$ that is even and looks like $(\cdots) \cdots 22222$ in base 3  and we divide it by 2 we will get a number which in base 3 looks like $(\cdots) \cdots 111$ which is great from a perspective of reducing using base 3. Observing that $(\cdots) \cdots 22222$ when we divide by 2 gives us  $(\cdots) \cdots 111$ is essentially the same as noting that $\cdots 111=-1/2$ in the 3-adics (or equivalently that $1/2$ is $0.111 \cdots$ in base 3). The situation will be similar in other bases since the $p$-adic expansion of $-1$ is just the digit $p-1$ repeating, and one will get corresponding reductions based on the behavior of those fractions. We have other similar paths corresponding to other fractions.  
 
The third section establishes the main theorem. This will use essentially the same method as in the second section but our set of primes is now larger. We will look at the set of primes $2,3,5,7,13,17$ and use those as our bases. As in the proof of the weaker theorem we will also need a few specific results about certain other moduli, in particular $p=11$ and $p=19$.

%The fourth section will establish some general forms for what $p$-adic expansions give rise to inequalities of this sort. 

The fourth section examines how far we can push our framework if certain simplifying assumptions are made and will implicitly use the ideas from the fourth section. In particular, in all the results in the second and third sections we need to engage in specific case checking for certain small values of $v_p(n+1)$ as well checking certain finite modulo conditions. This section will look at how for how low a value of $\alpha$ one can still get $S(\alpha,n)$ if one is allowed to assume that $v_p(n+1)$ is as large as we want for a finite collection of $p$ of our choice. 

The fifth section looks at more broad connections to various questions about $p$-adic expansions and related questions.  The sixth section is an appendix where the details of some proofs are included.

\section{Preliminaries}

We will write $S(\alpha,n)$ to denote the statement ``If for all $1<k<n$, $\cpx{k} \leq \alpha \log k$ then $\cpx{n} \leq  \alpha \log n$.'' Note that for some values of $\alpha$ and $n$, this statement will be vacuously true if the initial hypothesis is false.  

Note that, for a fixed $\alpha$, proving $S(\alpha,n)$ for all
$n$ would constitute proving the inductive step in a proof that
$C_M \leq \alpha$.  If for a given $\alpha$ we can show that $S(\alpha,n)$ holds for all $n$,
we can conclude that $C_M \leq \alpha$.  This will be our plan of attack. 

Since we have $\frac{\cpx{1439}}{\log 1439}=\frac{26}{\log 1439}$, we cannot hope to get a result better than $\alpha = \frac{26}{\log 1439}$. Under this basis, we will  restrict $\alpha$ to the interval $I_0= [\frac{26}{\log 1439}, \frac{3}{\log 2})$. The upper end of the interval $I_0$ is chosen using that  we already have just from the standard base 2 result that $S(n,\alpha)$ holds for all $n>1$ and $\alpha = \frac{3}{\log 2}$.  Restricting $\alpha$ to this range will allow us to simplify the presentation and proofs of some results.  On occasion, we will implicitly use that $\alpha$ is in $I_0$ even when we have not stated so explicitly. We'll write $I_1 = [\frac{26}{\log 1439}, \frac{5}{\log 2 + \frac{1}{2}\log 3})$, and write $I_2 = [\frac{26}{\log 1439}, \frac{25}{6\log 2 + 2\log 3})$.  Note that $I_0$ contains $I_1$ which contains $I_2$. We will sometimes restrict further the range of $\alpha$ where we show  $S(n,\alpha)$ to these two intervals. In those cases, we will do so after we have proving $S(n, \alpha)$ for all $\alpha$ to the right of that interval's end point. This will allow us to simplify some of the theorems and proofs, since we will be able to avoid breaking into as many cases; some results would have many distinct cases if we looked at all $\alpha$ in the $I_0$ range. 

We will write $v_p(m)$ to be the largest integer $k$ such that $p^k|m$. On occasion, we will abuse this notation and use it when $p$ is not prime.  We will, in general, use $v_p$ without an argument to denote $v_p(n+1)$. We will write $F(x)$ to be the maximum of $\cpx{n}$ over $n \leq x$. We will write $[p,q]n$ to be $\frac{n-p}{q}$ and note that when this is an integer we have $$\cpx{n} \leq \cpx{p}+\cpx{q} + \cpx{\frac{n-p}{q}} \leq F\left(\frac{n}{q}\right)+\cpx{p} + \cpx{q}.$$ We will write compositions of the $[p,q]$ operator in the obvious way. For example, we have $$[1,2][2,3]n=\frac{\frac{n-2}{3}-1}{2}.$$ Similarly, we will write $[p,q]^k(n)$ to denote repeating $[p,q]$ $k$ times. For example, if $n \equiv 7$ (mod $8$), we can write $n = 2(2(2([1,2]^3n)+1)+1)+1$ and would have then $\cpx{n} \leq \cpx{[1,2]^3n} + 9$. 

We will for convenience set $\cpx{0}=0$.

We will only use $[a,b]$ when $a$ is a non-negative integer and $b$ is a positive integer. Given non-negative $x_i$ and positive $y_i$ for $1 \leq i \leq k$ we say that $$[x_1,y_1][x_2,y_2] \cdots [x_k,y_k]m$$ is a \emph{path} or \emph{reducing path} or \emph{reduction} of length $k$.  We say that a path is \emph{valid} if  $$[x_1,y_1][x_{2},y_{2}] \cdots [x_k,y_k]m$$ is a positive integer.

We can use this notation to make our earlier observation about the behavior of $-1$ in the $p$-adics more rigorous. Let $F$ be a field and let $b$ be a non-zero element of that field, then we note that $$[b-1,b](-1)=-1.$$ This amounts to observing that when we repeatedly subtract $b-1$ and divide by $b$ this does not alter the values of $v_p$ for any prime $p$ where $(p,b)=1$. Equivalently, 
if we have integers $m$ and $b$, and $(b,m)=1$, then $[b-1,b](-1) \equiv -1 $ (mod $m$). 

%The following lemma is straightforward but is worth noting explicitly:

%\begin{lemma} If $[x_j,y_j][x_{j+1},y_{j+1}] \cdots [x_k][y_k]n$ is a valid path and $$\beta \log \prod_{i=1}^k y^k \geq
% \sum_{i=1}^k \cpx{x_i} + \cpx{y_i},$$ then   for all $\alpha 
%\geq \beta$ we have $S(\alpha,n)$.
%\end{lemma}
%\begin{proof} Assume as given. We have two cases to consider 
% $[x_j,y_j][x_{j+1},y_{j+1}] \cdots [x_k][y_k]n > 1 $ and 
% $[x_j,y_j][x_{j+1},y_{j+1}] \cdots [x_k][y_k]n=1$. In the
% first case, we may write $n$ in terms of  $[x_j,y_j][x_{j+
%1},y_{j+1}] \cdots [x_k][y_k]n$ in the obvious way and we the
% have $$\cpx{n} < F(\frac{n}{\prod y_i}) + \sum_{i=1}^k 
%\cpx{x_i} + \cpx{y_i} \leq \alpha \log (\frac{n}{\prod y_i}) +
% \sum_{i=1}^k \cpx{x_i} + \cpx{y_i} . $$
% Now note that $$\alpha \log (\frac{n}{\prod y_i}) + \sum_{i
% =1}^k \cpx{x_i} + \cpx{y_i} = \alpha \log n + \sum_{i=1}^k
% \cpx{x_i} + \cpx{y_i} - \log \prod y_i \leq \alpha \log n.$$
% The case when the path terminate in $1$ is similar but we 
%instead build $n$ up using $y_i$.
% \end{proof}

We will write that we have \emph{used} a path $[x_1,y_1][x_2,y_2] \cdots [x_k,y_k]m$ to mean we are writing $m$ in terms of $[x_1,y_1][x_2,y_2] \cdots [x_k,y_k]m$. Using a path will always carry with it the implicit claim that the path is valid.

We say that a specific path written down without any variables is \emph{fixed}. For example, $[1,6]7$ is a valid, fixed path. $[2,4][1,3]16$ is a fixed, invalid path (since $\frac{3}{4}$ is not an integer). We will refer to sets of paths written down in terms of $v_i$ as \emph{unfixed paths} or \emph{variable paths}. For example, $[1,2]^{v_2-1}n$ is a variable path which is valid for all $n$.

Also, note that for any prime $p$ we have that $[p-1,]^{v_p}n \not \equiv p-1$ (mod $p$) since if it were we would have $n$ ending in $v_p +1$ digits that  are all $p-1$ when written in base $p$.  We will refer to \emph{burning} a prime $p$ to mean doing $[p-1,p]$.

\begin{lemma} If $n$ is composite, $S(\alpha,n)$ holds for all $\alpha$.

\label{compositeinduction}
\end{lemma}
\begin{proof} Assume $n$ is composite, and assume moreover, that for all $1<k<n$ we have $\cpx{k} \leq \alpha \log k$.
We may write $m|n$ where $1 < m < n$. Thus, $\cpx{n} \leq \cpx{m}+\cpx{\frac{n}{m}} \leq \alpha \log m + \alpha \log \frac{n}{m}
= \alpha \log n$.
\end{proof}

\begin{lemma} For all $1<n<2 \cdot 10^6$, we have $\cpx{n} \leq \frac{26}{\log 1439} \log n$.
\end{lemma}
\begin{proof} This is just straightforward computation.
\end{proof}

The next lemma is a straightforward calculation. Although it looks technical, it essentially just says that we can bound a number's complexity by writing it in terms of repeatedly subtractions and divisions. 

\begin{lemma}\label{L:iterated}
Assume that $k=[x_\ell,y_\ell]\cdots[x_1,y_1]n\ge1$ is a valid reduction and that for $1<k<n$ 
we have $\cpx{k}\le\alpha\log k$, then 
$$\cpx{n}\le \alpha
\log\frac{n}{y_1 y_2\cdots y_\ell}+\sum_{j=1}^\ell(\cpx{y_k}+\cpx{x_k}) =  \alpha \log n + \sum_{j=1}^\ell(\cpx{y_k}+\cpx{x_k} - \alpha \log y_j)   .$$
\end{lemma}

\begin{proof}
We assume that $[x_\ell,y_\ell]\cdots[x_1,y_1]$ is a valid reduction of $n$, and we will induct on $\ell$ to prove the above inequality. The equality on the right hand of the above is trivial.

First assume that $\ell=1$. Then $k=[x_1,y_1]n \geq 1$. Therefore $n=ky_1+x_1$.

If $k=1$, $\cpx{n}\leq  \cpx{y_1}+\cpx{x_1}$. Since $x_1 \geq 0$ and 
$k \geq 1$ we have $n \geq y_1$ so that $\log\frac{n}{y_1} \geq 0$, and the result follows.

If $k>1$, since $y_1\ge2$,  we have $n>k$ and we have
\[\cpx{n} \leq \cpx{y_1}+ \cpx{x_1}+\cpx{k} \leq 
\Vert y_1\Vert+\Vert x_1\Vert+\alpha\log k.\]
We have also $n\ge ky_1$ so that $k\le n/y_1$ and the result follows. 

We proceed by induction. Let $m=[x_1,y_1]n$, by the hypothesis of induction
\[\cpx{m}\leq \alpha
\log\frac{m}{y_2\cdots y_\ell}+\sum_{j=2}^\ell(\cpx{y_k}+\cpx{x_k}).\]
We have $n=my_1+x_1$, therefore
\[\cpx{n}\leq \cpx{y_1}+\cpx{x_1}+\cpx{m}.\]
Since $n\ge my_1$ we obtain
\[\log\frac{m}{y_2\cdots y_\ell}\le \log\frac{n}{y_1y_2\cdots y_\ell}.\]
Combining our inequalities we get 
\[\cpx{n} \leq \cpx{y_1}+ \cpx{x_1}+\alpha
\log\frac{n}{y_1y_2\cdots y_\ell}+\sum_{j=2}^\ell(\cpx{y_k}+\cpx{x_k}),\] which is what we wanted to prove. 
\end{proof}

Lemma \ref{L:iterated} essentially amounts to a general process of writing $n$ in terms of smaller values of the $[p,q]$ operator. Note that the straightforward method of calculating $\cpx{n}$ is not the same as running through all possible valid options in applying the above lemma and then taking the smallest value. Our lemma above essentially does not take into account the small amount of subtraction we are doing actually does give us a smaller number to work with.  Equivalently, the above lemma ignores that when building up a number, additions increase the size of the number slightly. 

Lemma \ref{L:iterated} does not include the case $[x_\ell,y_\ell]\cdots[x_1,y_1]n=0$.
In this case the conclusion is not necessarily true with the difficulty connected to the fact that $\log 0$ is not defined. 

Note that this problem happens only for a unique particular 
value of $n$ if we fix the chain of operators because 
\[x_\ell=k=[x_{\ell-1},y_{\ell-1}]\cdots[x_1,y_1]n\]
implies 
\[n= y_1(\cdots(y_{\ell-2} (y_{\ell-1}k+x_{\ell-1})+ x_{\ell-2})\cdots)+x_1.\] We will often need to check values of this sort separately.

There are two broad categories of results required for the main theorem. The first set, ``definite''
results, are statements of the form $S(\alpha,n)$ for some range of $\alpha$ and for $n$ satisfying some finite
set of  conditions of the form $n \equiv k$ (mod $m$) or within some fixed finite range and always using valid fixed paths.  These will most of the time, be a statement of the form $v_p \geq c$ for some specific prime $p$ and value $c$. These cannot be used by themselves to improve upon the $\cpx{n} \leq 3\log_2 n$ bound since they will always have exceptional moduli. Generally, the definite results will have a modulo condition that will be of the form $n \not\equiv -1$ (mod $m$) for some $m$.

The second type of results we need, ``indefinite'', are results that look at $v_p(n+1)$ for various fixed primes $p$ and use variable paths. These results
can be thought of as being equivalent to various statements about the p-adic representations of certain specific fractions. Each indefinite result represents one possible reducing path which will give rise to one of our inequalities. Essentially, what we show is that given that the $v_p$ are large enough, at least one of the reductions corresponding to a indefinite result will be efficient enough for our purposes. The definite results will then handle the set of $n$ where the $v_p$ are too small to allow the reductions in the indefinite results.  

In general, the definite results are not very interesting but are a necessary foundation. They require many individual cases but demonstrate little in the way of actual structure. For convenience, we will label definite results as lemmata. We will label indefinite results as propositions when stated in terms of general $\alpha$ in a large interval. When we restrict these propositions
to  $\alpha > K$ will be labeled as corollaries. Some of the definite results are very similar, and in those cases we will restrict some of them to an appendix.  We will reserve the word Theorem for results of the form  ``For all $n \geq 1$, $\cpx{n} \leq \alpha \log n$
or of the form ``For sufficiently large $n$, $\cpx{n} \leq \alpha \log n$.''

There is one subtlety that is a problem for both definite and indefinite  results: there may be a small finite set of cases that one needs to check where the path in question is  not valid due to hitting 0, as discussed above. For example, if one wants to use the path $[1,3]n$ on all positive integers $n \equiv 1$ (mod $3$), then one will run into problems at $n=1$. Thus one will frequently need to check that the relevant small values of $n$ satisfy one's desired bound when the path fails to be valid. Unfortunately, for indefinite results, the same problem arises  but one does not, in general, have a finite list to test. The solution here is to insist on additional congruence restrictions on $n$ which ensure that at no time has one reduced too far. The details of this should be made clear in the first few results of where they come up, but we will not include these details in all the proofs.

\section{The weak version of the theorem}

In this section we will prove Theorem \ref{weakupperbound}.

 Note that $\frac{5}{\log 2 + \frac{1}{2}\log 3} = 4.204 \cdots$ whereas simply using base 2 yielded a constant of $\frac{3}{\log 2}= 4.328 \cdots$. Thus, this weaker version of our main theorem is already stronger than the previously known bound.
 
 We will now derive the inequalities  needed to prove this result: 

\begin{lemma} For $n>1$, $\alpha \geq \frac{6}{\log 6} = 3.34886 \cdots$ and $S(\alpha,n)$ fails, then we have $v_2 \geq 1$ and $v_3 \geq 1$.
\label{v2v3=0}
\end{lemma}
\begin{proof}
Direct computation establishes this for $n=2$ or $n=3$. We may thus assume that $n \geq 4$. If $v_2=0$ then we have $2|n$ and so our earlier
lemma about composites apply. Thus, we may assume that $v_2 \geq 1$. Assume then that $v_3 =0$. If $n \equiv 0$ (mod 3), then we may again use our composite lemma. Thus, we may assume that $n \equiv
1 $ (mod 3). Thus we can use the path $[1,6]n$ and since $$\frac{\cpx{6}+\cpx{1}}{\log 6}= \frac{6}{\log 6}$$ we have
$n= 1 + 6\frac{n-1}{6}$ and so $$\cpx{n} \leq \cpx{1} + \cpx{6} + \cpx{\frac{n-1}{6}} \leq 1+5 + \alpha \log \frac{n}{6} = 6- \alpha \log 6  + \alpha \log n \leq \alpha \log n.$$

\end{proof}

We will not go through this level of detail in later lemmata but will instead simply note the relevant paths and generally leave the raw calculation out.

\begin{lemma} For $n>1$, if $\alpha \geq \frac{5}{\log 4} = 3.60 \cdots $ and $S(\alpha,n)$ fails, then $v_2 \geq 2$. 
\label{v2=1}
\end{lemma}
\begin{proof}  Assume that $v_2 =1$. The hypothesis $v_2=1$ allows us to assume that $n \equiv 1$ (mod 4) and $n>1$. Thus we use the path $[1,4]n$.
\end{proof}

\begin{lemma} For $n>1$ if $\alpha \geq \frac{5}{\log 4}$ and  $S(\alpha,n)$ fails then either $v_2 \geq 3$ or $v_3 \geq 2$.
\label{v2=2andv3=1}
\end{lemma}
\begin{proof} Assume as given, and assume that $v_2=2$ and $v_3=1$.  We thus have $n \equiv 2$ (mod 3). and $n \equiv 3$ (mod 8).
We then we have $n \equiv 3$ or $11$ (mod 16). In the first case, we may take the reduction path $[1,8][1,2]n$. We can  check that
$\cpx{8} + \cpx{2} + \cpx{1} + \cpx{1} \leq \frac{5}{\log 4}\log (16)$.  Similarly, in the second case, we have either $n \equiv 2$ (mod 9) or $n \equiv 5$ (mod 9).
If $n \equiv 2$ (mod 9), we have either $n=11$ which works
or we have the path $[1,18][1,4]1,2]n$ which can be easily seen to be valid and works for this value of $\alpha$.
If $n \equiv$ (5 mod 9), we need to break it down into two further cases: $n \equiv$ 11 (mod 32) or
$n \equiv 27$ (mod 32). If $n \equiv 11$ (mod 32) we may
use $[1,12][1,4][1,2]n$ which is sufficient.
If $n \equiv 27$ (mod 32) we either have $n=59$ or we may use $[1,6][1,6][1,4][1,2]n.$
\end{proof}

We also have the following slightly stronger lemma that we don't need but is worth noting:

\begin{lemma} For $n > 1$, $\alpha \geq \frac{11}{\log 24}= 3.46124 \cdots$ and  $S(\alpha,n)$ fails, then $v_2 \geq 2$.
\label{auxlemmav2=1}
\end{lemma}
\begin{proof}
We may assume that $n \equiv 2$ mod 3 since otherwise $v_3=0$, which would trigger Lemma \ref{v2v3=0}. Since $v_2=1$ we must have $n \equiv 1$ (mod 4) and so we have $n \equiv 1$ (mod 8). or $n \equiv 5$ (mod 8).
If $n \equiv$ (1 mod 8), we may use $[1,8]n$.
If $n \equiv 5$ (mod 8), either $n=5$ or we may use $[1,6][1,4]n$ .
\end{proof}

\begin{lemma} For $n>1$ if  $\alpha \geq \frac{15}{\log 60} = 3.66359 \cdots$ and $S(\alpha,n)$ fails, then $v_5 \geq 1$.
\label{lemma7}
\end{lemma}
\begin{proof}
We may assume $n \equiv 1,2,$ or $3$ (mod 5). We may also assume that $v_2 \geq 2$ and $v_3 \geq 1$.\\

Case {\bf I}: $n \equiv 1$ (mod 5). Then we use the path $[1,3][1,10]n$. \\

Case {\bf II:} $n \equiv 2$ (mod 5). We have two subcases, the case of $b = 1$ and $b >1$. \\

Case {\bf IIa:} $v_3=1$. We may then assume $v_2 \geq 3$ (or the previous lemma triggers). So we have $n \equiv 7$ (mod 8)
and $n \equiv 2$ or $5$ (mod 9). If $n \equiv 2 $ (mod 9) we may use $[2,45]n$. If $n \equiv 5$ (mod 9), and $n \equiv 7$ (mod 16) we then use
$[1,6][1,4][1,2][2,15]n$ If $n \equiv 5$ (mod 9) and $n \equiv 15$ (mod 16)
we may then use $[1,8][1,2][2,15]n$\\

Case {\bf IIb:} $v_3 \geq 2$. So we have $n \equiv 8$ (mod 9). \\
We may then use $[1,6][2,15]n$.\\

Case {\bf III:} $n \equiv 3$ (mod 5).
We may then use $[1,3][1,10][1,2]n$. \\
\end{proof}

One might want to improve the above lemma to  allow $\alpha \geq \frac{5}{\log 4}$ but there does not seem to be any obvious path to do so.  IIa and III prevent this improvement. Improving this lemma might allow us to get correspondingly tighter bounds on the next few lemmata. \\

We now begin our indefinite results:

\begin{proposition}
 If $S(\alpha,n)$ fails, $v_2 \geq 1$, $v_3 \geq 1$ and $v_5 \geq 1$, then
$$(3-\alpha \log 2)v_2 > (\alpha \log 3 -4)v_3 + (\alpha \log 2 -2).$$
\label{2then3basic}
\end{proposition}

\begin{proof}

We may assume that $v_2 \geq 1$, $v_3 \geq 1$, $v_5 \geq 1$ and assume that $S(\alpha,n)$ fails for some
$\alpha$. We note that $[0,2][1,2]^{v_2}n$ will have $v_3$ 1s at the end of its base-3 expansion. Thus, we may use the reduction $$\cpx{n} \leq \cpx{[1,3]^{v_3}[0,2][1,2]^{v_2}n} + 3v_2 + 4v_3 + 2.$$
So we have $$\alpha\log n < \cpx{n} \leq \alpha\log (\frac{n}{3^{v_3}2^{v_2+1}}) + 3v_2 +4v_3 +2,$$ and so  we have 
$$\alpha \log n < \alpha \log n - (\alpha\log 3)v_3 - (\alpha \log 2)(v_2+1) +3v_2+4v_3+2$$
which forces
$(\alpha \log 3 -4)v_3 + \alpha-2 < (3-\alpha \log 2)v_2$.

\end{proof}

The above proof essentially amounts to using the fact that $\frac{-1}{2}$ in the $3$-adics is $\cdots 111$ (equivalently, that $1/2$ has expansion $0.111\cdots$ in base 3), along with the facts that $n \equiv -1$ (mod) $3^{v_3}$, and $[1,2](-1)=-1$. Thus, repeatedly applying $[1,2]$ doesn't alter how the number is behaving mod $3^{v_3}$, but $3$-adically $-1=\cdots 2222$ to the resulting number must have at least $v_3$ 2s in its base 3 expansion, which when we do a $[1,4]$ then turns into $v_3$ 1s at the start of the base 3 expansion. 

We can improve the above proposition since at the end of the reduction we can use either $[0,3]$ or $[2,15]$ since we know that the number at that stage is not $1$ (mod 3) since we have exhausted all the $v_3$, and $-1/2 \equiv 2$ (mod 5). We note that, that the worse case is scenario here is that we use the $[2,15]$ (which is always a weaker reduction than $[0,3]$  We then obtain:

\begin{proposition}
\label{3then2shorttail}
If $S(\alpha,n)$ fails, $v_2 \geq 1$, $v_3 \geq 1$ and $v_5 \geq 1$, then
$$(3-\alpha \log 2)v_2 > (\alpha \log 3 -4)v_3 + \alpha \log 6-5.  $$
\end{proposition}

Here the $\alpha \log 6 -5$ term comes from adding $\alpha \log 2 -2$ and $\alpha \log 3-3$
This result is slightly stronger than the earlier result, and will be necessary in the proof of the weak version of the theorem since we want a linear inequality in $\alpha$ which has a positive constant term for the range of $\alpha$ we care about.

Similarly, we obtain: 
\begin{proposition} Assume that  $v_2 \geq 3$, $v_3 \geq 1$, $v_5 \geq 1$,  and $S(n,\alpha)$ fails for some $\alpha \in I_0$. Then:  
$$\left(5-\alpha \log 3 \right)v_3 > \left(2\alpha \log 2 - 5\right)\lfloor\frac{v_2-1}{2}\rfloor + \alpha \log 6 -6 $$
\label{3then2}
\end{proposition}
\begin{proof} We have $\cpx{n} \leq \cpx{[2,3]^{v_3}n}+5v_3$. Note that $k=[2,3]^{v_3}n$ is either $1$ (mod $3$) or $0$ (mod $3$). If $k$ is $0$ (mod 3), then we have $$\cpx{k} \leq \cpx{[1,4]^{\lfloor v_2/2 \rfloor}[0,3]k} + 5\lfloor v_2/2 \rfloor + 3.$$ This gives rise to the inequality $$\left(2\alpha \log 2 - 5\right)\lfloor\frac{v_2}{2}\rfloor + \alpha \log 3 -3 < \left(5-\alpha \log 3 \right)v_3.$$

If $k$ is $1$ (mod 3) we instead have $$\cpx{k} \leq \cpx{[1,4]^{\lfloor (v_2-1)/2 \rfloor}[1,6]k} + 5\lfloor (v_2-1)/2 \rfloor + 6.$$ This gives rise to the inequality $$\left(2\alpha \log 2 - 5\right)\lfloor\frac{v_2-1}{2}\rfloor + \alpha \log 6 -6 < \left(5-\alpha \log 3 \right)v_3.$$ We note that the first inequality is always stricter than the second as long as we have $\alpha < \frac{3}{\log 2}$ which we can assume since otherwise we cannot have $S(n,\alpha)$ fail at all. 

Thus, in either case we have the desired inequality. Note that $v_5 \geq 1$ ensures that the paths in question are always valid.
\end{proof}
The above result can be thought of as essentially arising from $-1/3$ having $2$-adic representation $01$ repeating; note that $[1,4]$ is the same as $[0,2][1,2]$ which corresponds to our $01$ repeating part in the 2-adic expansion and is equivalent to saying that if a number $k$ divisible by 3 ends in $m$ 1s in base 2 then $k/3$ will have look like $(\cdots)\cdots 0101010101$.

One can easily see from the above propositions that one has a contradiction if $S(\alpha,n)$ fails and $\alpha \geq \frac{5}{\log 2 + \frac{1}{2}\log 3}$. Thus one can conclude simply from the above using base 2 and base 3 that for all $n>1 $ we have $$\cpx{n} \leq \frac{5}{\log 2 + \frac{1}{2}\log 3} \log n.$$

A few remarks: We have three types of ``gadgets'' that we will use in our indefinite reductions.  The above proof shows examples of two of those types. The first type is an ``initial gadget'' which allows us to go from doing an inefficient reduction at the beginning of the form $[p-1,p]^{v_p}$ to an actually helpful reduction, generally but not always, of the form $[1,p+1]^x$ for some x dependent on some set of primes and exactly what form of gadget we used. In our above example of the reductions in Proposition \ref{3then2}, our initial gadget is the pair of reductions of either $[0,3]$ or $[1,6]$ depending. We will  write initial gadgets as [I]. 

Our next gadget is a ``final gadget''  which use a small but finite set of remaining primes at the end after a series of reductions. In Proposition \ref{3then2shorttail} our final gadget $[F]$ takes on the values either $[0,3]$ or $[2,15]$. 

Final gadgets are necessary because for many of our reductions, the initial gadget used will be somewhat inefficient, and so without the final gadget, the constant in the resulting linear inequality will be weak. Squeezing out a small amount of efficiency from the final gadgets allows us to mitigate and sometimes remove this issue. Unfortunately, the final gadgets are sometimes ugly and make our resulting inequalities harder to follow. Readers are encouraged to read through the final gadgets, but to be aware that they are not really structural in any deep sense.\\

In this context, we can rewrite the essence of the proof of Proposition \ref{3then2} by saying that we have the reduction $$[1,4]^e[I][2,3]^{v_3}n$$ where $[I]$ is the initial gadget that is either $[0,3]$ or $[1,6]$ and then we choose $e$ to do as much $[1,4]$ reduction as we can. Using the our earlier terminology,  we could say that the essence of  Proposition \ref{3then2} is to repeatedly burn $3$ followed by an initial gadget allowing us to do repeated $[1,4]$ reduction.   

There is a third type of gadget, a ``transition gadget'' which will be used later and does not show up in this simple proof. Essentially, sometimes we will exhaust doing one sort of efficient reduction, and then want to switch to doing a more efficient reduction.  An example of this below occurs in Proposition \ref{prop3} where after exhausting repeated $[1,3]$ reduction we go to $[1,5]$ reduction. We shall call these gadgets ``transition gadgets'' and denote them with $[T]$.  The most common transition gadget will be ones which allow us to use an extra $[0,2]$ reduction, allowing us to switch from doing $[1,q]$ reductions to $[1,2q-1]$ reduction. 
 We will discuss this gadget more when it occurs in Proposition \ref{prop3} below. 

For both the initial and transition gadget, maximizing their efficiency will be critical to having small enough constants.  The next section of this paper will address a related question: what happens when $v_2$ is assumed to be sufficiently large? This will allow us to address what happens if we can both ignore the need for definite results and ignore the efficiency of our gadgets and merely care about their existence (and thus ignore final gadgets completely).

\begin{comment}
Rewrite above using $i$ notation? Also note that we can use tailed version with [0,3], [0,2],[1,6] or [2,21] if include 7 bit. 
\end{comment}

The above illustrates the essential method. We will use the same method to prove the main theorem, but will do so by relying on a larger number of bases and and frequently will also require checking special cases as well. 

The basic method of many lemmata is similar enough to that above that we will not include all the details; in some cases, we will include the proofs in the Appendix.

\section{The proof of the main theorem}

In this section we will prove Theorem \ref{upperbounddirectimprovement}. 

\begin{lemma} If $\alpha \geq \frac{18}{\log 135} =3.669 \cdots$, and $S(\alpha, n)$ fails, then $v_3 \geq 2$.
\label{v3=1}
\end{lemma}
\begin{proof}
Given the earlier lemmata, we may assume that $v_2 \geq 3$. So $n \equiv 7$ (mod 8). Similarly,
we may assume that $n \equiv 4$ (mod 5). We may assume that $v_3 1$ and thus must have either $n \equiv 2$ (mod 9) or $n \equiv 5$ (mod 9).

Case {\bf I}: $n \equiv 2$ (mod 9)
In this case, we have either $n \equiv 2,11$ or $20$ (mod 27).
If $n \equiv 2$ (mod 27) then we use $[2,27]$. 
If $n \equiv 11$ (mod 27) then we use $[1,12][2,9]$. 
If $n \equiv 20$ (mod 28) then we depending on whether $[1,6][1,4][2,9]n$ is even or odd
we either use $[1,12][1,4][2,9]n$ or $[1,10][1,6][1,4][2,9]n$. \\

Case {\bf II}: $n \equiv 5$ (mod 9). So $n \equiv 5,14$ or $23$ (mod 27)
If $\cpx{n} \equiv 5$ (mod 27), then we may use $[1,18][2,3]n$.
If $n \equiv 14$ (mod 27) then we may use $[1,15][1,3][2,3]n$.
If $n \equiv 23$ (mod 27) then depending on the parity of $[1,12][1,6][2,3]n$ we use
either  $[1,24][1,6][2,3]n$ or $[1,10][1,12][1,6][2,3]n$.
\end{proof}

\begin{lemma} If  $\alpha \geq \frac{29}{\log 2304} = 3.7456\cdots$ and $S(\alpha,n)$ fails, then $v_2 \geq  6$,.
\label{v2=5}
\end{lemma}

\begin{proof} We may assume that $v_5 \geq 1$, and $v_3 \geq 2$.

We may write $n$ in terms of one of the following reductions depending on $k=[1,4][1,2]^{v_2-1}(n)$ (mod 4):

Case {\bf I}: If $k \equiv 1$ (mod 4), we may use the reduction $$[1,3][1,12][1,4][1,2]^{v_2-1}(n).$$ %Worst case tied  with [1,12][1,3]

Case {\bf II}: If $k \equiv 2$ (mod 4), we may use the reduction $$[1,10][1,6][1,3][1,4][1,2]^{v_2-1}(n).$$

Case {\bf III}: If $k \equiv 3$ (mod 4), we may use the reduction $$[1,3][1,10][1,6][1,4][1,2]^{v_2-1}(n).$$

Case {\bf IV}: If $k \equiv 0$ (mod 4), we may use the reduction $$[1,12][1,3][1,4][1,2]^{v_2-1}(n).$$

\end{proof}

The next two lemmata have very similar proofs and have the proofs included in the appendix. 
\begin{lemma} If $\alpha \geq  \frac{29}{\log 2304} = 3.7456 \cdots$ and $S(\alpha,n)$ fails, then $v_7 \geq 1$.
\label{v7=0}
\end{lemma}

Through similar logic we obtain:
\begin{lemma} If $\alpha \geq  \frac{29}{\log 2304} = 3.7456 \cdots$ and $S(\alpha,n)$ fails then $v_{13} \geq 1$.
\label{v13=0}
\end{lemma}

The proof is again the Appendix

\begin{lemma} Assume $\alpha \geq \frac{29}{\log 2304} = 3.7456  \cdots$ and that $S(\alpha,n)$ fails. Then $v_3 \geq 3$.
\label{v_3=2}
\end{lemma}
\begin{proof} Given the above, we may assume that $v_2 \geq 6$, $v_5 \geq 1$, $v_7 \geq 1$, and $v_{13} \geq 1$, We may assume that $v_3 =2$. Thus, we have either 
$[2,3][2,3]n \equiv 0$ or $1$ (mod 3).
First, assume that $[2,3][2,3]n \equiv 0$  (mod 3) consider $k = [1,4]^3[2,9][2,3]n$. If $k$ is even we may use
$[0,4][1,4]^3[2,9][2,3]n$. If $k$ is odd we may use one of  $[1,4]^4[2,9][2,3]n$ and
 $[1,10][1,2][1,4]^3[2,9][2,3]n$.

Now assume that $[2,3][2,3]n \equiv 1$ (mod 3). Look at $k=[1,4]^i[1,6][2,3][2,3]n$, where $i$ is as large as possible (and thus is at least 3). Note that $k \not \equiv 1$ (mod 4). Consider then the possibilities for $k$ (mod 4). \\ 

Case {\bf I}: $k \equiv 2$ (mod 4) Then we may use $[1,14][0,2][1,4]^3[1,6][2,3][2,3]n$.\\

Case {\bf II}: $k \equiv 3$ (mod 4). Then we may use $[1,7][1,10][1,2][1,4]^3[1,6][2,3][2,3]n$.\\ 
 
Case {\bf III}: $k \equiv 0$ mod 4. Then we may use $[1,13][0,4][1,4]^3[1,6][2,3][2,3]n$.

\end{proof}

\begin{lemma}
If  $\alpha \geq \frac{41}{\log 55296} = 	3.7544 \cdots$ and  $S(\alpha,n)$ then $v_2 \geq 8$.
\label{v_2=8}
\end{lemma}
\begin{proof}
We may assume that $v_3 \geq 3$, $v_5 \geq 1$, and that $v_7 \geq 1$.

Set $k=[1,4][1,2]^{v_2-1}n$. Consider $k$ (mod 4). We have four cases:

Case {\bf I}: $k \equiv 0 $ (mod 4) we may use $[1,3][1,9][1,16][1,2]^{v_2-1}n$\\
% 3.7544
% 21 + 9 + 7 + 4 = 41

Case {\bf II}: $k \equiv 1 $ (mod 4) we may use $[1,9][1,12][1,4][1,2]^{v_2-1}n$\\
% Same as above % 3.7544 

Case {\bf III} $k \equiv 2 $ (mod 4) we may use $[1,3][1,9][1,10][1,8][1,2]^{v_2-1}n$.\\
%3.75102

Case {\bf IV}: $k \equiv 3 $ (mod 4) we may use 
$[1,9][1,10][1,6][1,4][1,2]^{v_2-1}n$.\\

% 	$$3.75102 TIED WITH ABOVE. 
Note that the first two cases above are the current worst case scenarios. 
\end{proof}

We can now use the above lemmata and a similar sort of logic to obtain:
\begin{lemma}
If $\alpha \geq \frac{41}{\log 55296} = 	3.7544  $, $S(\alpha,n)$ fails, then $v_{11} \geq 1$.
\label{v11=0}
\end{lemma}

\begin{lemma} Assume $\alpha \geq \frac{41}{\log 55296} = 	3.7544 \cdots$ and that $S(\alpha,n)$ fails. Then $v_3 \geq 4$.
\label{v_3=3}
\end{lemma}
\begin{proof} Given the above, we may assume that $v_2 \geq 7$, $v_5 \geq 1$, $v_7 \geq 1$, and $v_{13} \geq 1$
$[2,3]^3n \equiv 0$ or $1$ (mod 3).

Assume that $[2,3]^3n \equiv 1$ (mod 3) (the situation where it is 0 (mod 3) is nearly identical and a little easier). Look at $k=[1,4]^i[1,6][2,3][2,3]n$, where $i$ is as large as possible (and thus is at least 3). Note that $k \not \equiv 1$ (mod 4). Consider then the possibilities for $k$ (mod 4). \\

Case {\bf I}: $k \equiv 2$ (mod 4) Then we may use $[1,13][1,14][0,2][1,4]^3[1,6][2,3]^3n$ where $[A]$ is one of $[1,26]$, $[0,4]$ or $[1,10][0,2]$.\\

Case {\bf II}: $k \equiv 3$ (mod 4). Then we may use $[B][1,10][1,2][1,4]^3[1,6][2,3]^3n$ where $[B]$ is one of $[1,26][0,2]$, $[0,4]$ or $[1,13][1,14]$.\\
 
Case {\bf III}: $k \equiv 0$ mod 4. Then we may use $[1,13][0,4][1,4]^3[1,6][2,3]^3n$.

\end{proof}

We may now prove:

\begin{lemma}
If  $\alpha \geq \frac{41}{\log 55296} = 	3.7544 \cdots$ and  $S(\alpha,n)$ then $v_2 \geq 10$.
\label{v_2=10}
\end{lemma}
\begin{proof}
We may assume that $v_3 \geq 4$ and $v_5 \geq 1$. 

Set $k=[1,4][1,2]^{v_2-1}n$. Consider $k$ (mod 4). We will write $[A]$ to be one of $[0,2]$ or $[1,22]$ We have four cases:\\

Case {\bf I}: $k \equiv 0 $ (mod 4) we may use $[1,9][1,9][1,16][1,2]^{v_2-1}n$.\\

Case {\bf II}: $k \equiv 1 $ (mod 4) we may use $[A][1,3][1,9][1,12][1,4][1,2]^{v_2-1}n$.\\

Case {\bf III} $k \equiv 2 $ (mod 4) we may use $[A][1,3][1,9][1,10][1,8][1,2]^{v_2-1}n$.\\

Case {\bf IV}: $k \equiv 3 $ (mod 4) we may use $[A][1,3][1,9][1,10][1,6][1,4][1,2]^{v_2-1}n$.\\
%This is the worst case scenario here when use [0,2]. Is then 	%3.729 about. 
\end{proof}
% NOTE: When do general form ending in repeated [1,9] can use this [A] tail to handle the odd 3 case a little more elegantly. 

\begin{lemma}
\label{v19=0}

If $\alpha \geq \frac{41}{\log 55296} = 	3.7544 \cdots$  and $S(\alpha,n)$ fails, then $v_{19} \geq 1$.

\end{lemma}

The proof is in the appendix. % NOTE THIS CAN BE POSSIBLY MOVED UP EARLIER. Never need more than $v_2 \geq 8$. Can also possibly tighten by replacing the inefficient [6,7], [12,13] and [4,5] reductions instead with good reductions. 

\begin{lemma} Assume that  $\alpha \geq \frac{41}{\log 55296} = 3.7544 \cdots$. Assume that $S(\alpha,n)$ fails. Then $v_5 \geq 2$.
\end{lemma}
\begin{proof} We may assume that $v_3 \geq 3$ and that $v_p \geq 1$ for $p= 7$, $11$, $13$, $19$.  We may assume that $v_2 \geq 8$ (In fact may use that $v_2 \geq 9$ but we will not need it here.)

We now assume that $v_5 =1$, and consider cases:

Case {\bf I}: $n \equiv 4$ (mod 25). We may then take
$[1,6]^3[4,25]n$.\\

Case {\bf II:} $n \equiv 9$ (mod 25). We may then take
$[1,6]^3[4,25][1,2]n$.\\

Case {\bf III:} $n \equiv 14$ (mod 25). We have two scenarios depending on whether $v_3 -1\geq v_2$ or not. If  $v_3 -1 \geq v_2$, then we have $v_3 \geq 8$ and so we may take $[1,6]^8[4,25][2,3]n$. 

If $v_3 -1 \leq v_2$ we  may instead use $[1,16][A][1,6]^i[4,25][2,3]n$. Here $i=v_3-1$ and then $[A]$ is either $[0,3]$ or $[2,33]$. We are using here that $[1,6]^i[4,25][2,3]n$ cannot be $1$ mod 3. \\

Case {\bf IV:} $n \equiv 19$ (mod 25). The argument for Case IV is nearly identical to that of Case III. 

%We may then take

%$[1,16][A][1,6]^3[4,25][1,2]^2(n)$ Here $[A]$ is again $[0,3]$ or $[2,33]$. Note that here the worst case scenario is in the last case where we do 
%$[1,16][2,33][1,6]^3[4,25][1,2]^2(n)$. Need actually if put in explicitly to do separation trick as above. 

\end{proof}

This next lemma is not strictly speaking needed, but it is useful when checking the larger base reductions, and was used to
make some of the computations more efficient.
\begin{lemma} If $p$ is an odd prime and $p \leq 10^6$, and $p|n-1$ then $S(\alpha,n)$ for all
$\alpha \geq \frac{9}{\log 12} =3.62 \cdots$ .
\label{lemma17}
\end{lemma}

\begin{proof}

Assume as given. Note that by direct computation we have $n < \frac{5}{2} \log_2 n$
for all $1 < n \leq 2 \cdot 10^6 +1$. So we may assume that $n$ is prime and $n > 2 \cdot 10^6 +1$. Earlier results prove
the desired claim for $p=2,3,5,7$ and we may assume that $n \equiv 2$ (mod 3) and $n \equiv 3$ (mod 4). If $p=11$ then we may use
the reduction $[1,6][1,22]n$. So we may assume that $p \geq 13$.

First consider the case when $p \equiv 2$ (mod 3). Then $p \geq 17$ and we have $$n= 2p\left(6[1,6]\left(\frac{n-1}{2p}\right) +1 \right) +1.$$
Thus, we have $\cpx{n} \leq F\left(\frac{n}{12p}\right) +\cpx{p}+7$. Thus,
we have $\cpx{n} \leq \alpha \log_2 \frac{n}{12p} + \frac{5}{2}\log_2 p +7 \leq \alpha \log_2 n - \alpha \log_2 p + \frac{5}{2} \log_2 p + 7 - \alpha \log_2 12$.
So $\cpx{n} \leq \alpha \log_2 (n) + (\frac{5}{2} - \alpha)\log_2 p + 7 - \alpha \log_2 12$. The right hand side
of this last inequality is bounded by $\alpha \log_2 n$ when $\alpha \geq \frac{\frac{5}{2}\log_2 17 + 7}{\log_2 17 + \log_2 12}$ which is less than
$\frac{9}{\log_2 12}$.

Now, consider the case when $p \equiv 1$ (mod 3). Then we have $$n= 2p\left(\frac{n-1}{2p}\right) +1$$ and can use
a $[1,6]$ reduction on $p$. Then we have $\cpx{n} \leq \cpx{[1,2p]n} + \cpx{[1,6]p} + 9 \leq \alpha \log_2 \frac{n}{2p} + \alpha \log_2 \frac{p}{6} + 9
= \alpha \log_2 n + 9 - \alpha \log_2 12 \leq \alpha \log_2 n$ with this last inequality valid when $\alpha \geq \frac{9}{\log_2 12}$
\end{proof}

We can with a little work present a version of Proposition \ref{3then2shorttail} which is tighter for the range of $\alpha$ we care about:

\begin{proposition}
\label{2then3}
Assume that  $S(\alpha,n)$ fails for some $\alpha \in I_0$. Assume further $v_2 \geq 1$, $v_3 \geq 1$, $v_5 \geq 1$, $v_7 \geq 1$,  $v_{11} \geq 1$, and $v_{13} \geq 1$. Then $$ (3-\alpha \log 2)v_2 > (\alpha \log 3 -4)v_3 + \alpha \log 2-2 +C_1(\alpha).$$ Here $C_1(\alpha)$ = $\alpha \log 2730 -28$ when $x \in [\frac{26}{\log 1439}, \frac{7}{\log 7})$, and $C_1(\alpha)= \alpha \log 390 -21$ when $\alpha \in [\frac{7}{\log 7}, \frac{3}{\log 2}).$

\end{proposition}
\begin{proof}Assume that $S(\alpha,n)$ fails for some $\alpha$. We have $$\cpx{n} \leq \cpx{[1,3]^{v_3}[1,4][1,2]^{v_2-1}n} + 3v_2 + 4v_3 + 2.$$ Set $k=[1,3]^{v_3}[1,4][1,2]^{v_2-1}n$ and note that $k$ is not $1$ (mod 3) since we have exhausted all the digits which were $1$ in base 3, since there were exactly $v_3$ of them.  We have then the following congruence restrictions on $k$:
\begin{itemize}
    \item $k \equiv 0$ or $2$ (mod 3). 
    \item $k \equiv 2$ (mod 5). This is because $[1,4][1,2]^{v_2-1}n \equiv 2$ (mod $5$), and the operator $[1,3]$ sends $2$ (mod 5) to $2$ (mod 5). Equivalently, this is because $-1/2 \equiv 2$ (mod 5). By similar logic we have:
    \item $k \equiv 3$ (mod 7).
    \item $k \equiv 5$ (mod 11).
    \item $k \equiv 6$ (mod 13). 
    
\end{itemize}
Note that this is sufficient to get that $k \geq 2502$. Now,if $k$ is $0$ (mod 3) then we can write $k$ itself in terms of either $[1,13][1,14][0,3]k$ or $[1,13][0,6]k$ depending on whether $k$ is even or odd. If $k$ is $2$ mod 3, then we can instead write $k$ in terms of $\ell=[2,15]k$, and then as before do either $[1,13][1,14]\ell$ or $[1,13][0,2]\ell$ to the end. Writing the complete reduction out we have the reduction $$[F][1,3]^{v_3}[I][1,2]^{v_2}n $$ where our initial gadget $[I]=[0,2]$ and our final gadget $[F]$ is one of $[1,13][1,14][0,3]$, $[1,13][0,6]k$, $[1,13][1,14][2,15]$, or $[1,13][0,2][2,15]$. We then have under the assumption that $S(\alpha,n)$ fails that \begin{equation} \alpha \log 3^{v_3}2^{v_2+1}   + f(\alpha) \leq 4v_3 +3(v_2-1)+ 5\end{equation} where $f(\alpha)$ is the contribution of the final gadget. Counting the contributions of each possible option for our final gadget and labeling them $f_i(\alpha)$ for $i=1,2,3,4$. We have then:
\begin{equation}
\begin{split}
    f_1(\alpha) &= \alpha \log 78 -14 + (\alpha \log 7 -7)\\
    f_2(\alpha) &= \alpha \log 78 - 14 \\
    f_3(\alpha) &=\alpha \log 78 -14 + + (\alpha \log 5 -7)\\
    f_4(\alpha) &=\alpha \log 78 -14 + + (\alpha \log 35 -14).
\end{split}
\end{equation}

We note that we do not need to include a given final gadget piece when reducing with a given $\alpha$ if the corresponding $f_j(\alpha)$ is negative. Thus we may take $$C_1(\alpha)= \min(\max(f_1(\alpha),0), \max\left(f_2(\alpha),0), \max(f_3(\alpha,0)), \max(f_4(\alpha),0)\right).$$ A straightforward calculation then determines that $C_1(\alpha)$ takes on the values claimed.  

\end{proof}

The level of detail we have gone in the proof above is higher than we will generally go into in similar results below.\\

From Proposition \ref{2then3} we have

% OUR relevant alpha is much larger than 7/log 7. 
\begin{corollary} Assume that  $\alpha \geq \frac{41}{\log 55296} = 3.7544 \cdots$, and that $S(\alpha,n)$. Then \begin{equation}\label{2then3corequation}v_2 > 0.31349v_3 + 5.03429. \end{equation}
\label{2then3cor}
\end{corollary}
\begin{proof} Simply plug in $\alpha \geq \frac{41}{\log 55296} = 3.7544 \cdots $ into Proposition \ref{2then3}. 
\end{proof}

\begin{proposition}
Assume that $S(\alpha,n)$ fails for some $\alpha$ and $v_2 \geq 2$, $v_3 \geq 1$, $v_5 \geq 1$, and $v_7 \geq 1$. 

$$(3-\alpha \log 2)v_2 > (\alpha \log 3- 4)v_3 + (\alpha \log 5 -6)v_5 + C_2(\alpha).   $$

Here $C_2(\alpha) = 2\alpha \log 2-4$ when $\alpha \in [\frac{26}{\log 1439}, \frac{4}{\log 3}]$, and $C_2(\alpha) = \alpha \log \frac{4}{3}$ when $\alpha \in [\frac{4}{\log 3}, \frac{3}{\log 2})$.  %NOTE YES, NO CONSTANT IS CORRECT HERE. 

\label{prop3}
\end{proposition}
\begin{proof}
We may assume that $v_2 \geq 2$, $v_3 \geq 1$ and $v_5 \geq 1$, $v_7 \geq 1$ and we then can use
the following reductions:

$$[1,5]^{v_5}[T][1,3]^{v_3-1}[0,2][1,2]^{v_2}n$$ where $[T]$ is the transition gadget which is either $[0,2][1,3]$ or $[0,2]$. Note that in this case we are corresponding $p$-adic expansions are that $-1/2$ is 1-repeating in the $3$-adics and that $-1/4$ is 1-repeating in the $5$-adics. We wish to see that this reduction is valid. Set $m=[1,2]^{v_2}(n)$. One has that $v_p(m)=v_p$ for any odd prime $p$, and particularly for $p=3$ and $p=5.$ We note that $m$ is even, and so the $k=[0,2]m$ is valid. This will now be a number which ends in $v_3$ 1s in base 3, and $v_5$ 2s in base 5. We then have $\ell= [1,3]^{v_3-1}k$ is valid and will still end in $v_5$ 2s in base 5. We want to be able to divide by 2 again, so we use our transition gadget of either $[0,2][1,3]$ or $[0,2]$ depending on whether $\ell$ is even or odd. The resulting number then ends in $v_5$ $1$s in base 5, and so the final $[1,5]^{v_5}$ is valid.   We use that $v_7 \geq 1$ to insure that the number we eventually reduce to is at least 1. 

One may then verify that the constant for $C_2(\alpha)$ gives each worst case scenario corresponding to the given range of $\alpha$.

\end{proof}

We will not in general, go through the same level of detail in later results of stating explicitly how all the base aspects are preserved except in some non-obvious cases. 

As before we gain an 
explicit version for our a given range of $\alpha$:

\begin{corollary}

If $S(\alpha,n)$ fails for some $\alpha \geq \frac{41}{\log 55296} = 3.7544 \cdots$ then we have \begin{equation}\label{2then3,5easyform eq}v_2 > 0.31349v_3 + 0.1069v_5+	2.71627.\end{equation}

\label{2then3,5easyform}
\end{corollary}

We also have:

\begin{proposition}

Assume that $S(\alpha,n)$ fails. Assume further that $v_2 \geq 1$, $v_3 \geq 3$, $v_5 \geq 2$, and $v_7 \geq 1$. Then

$$ (3-\alpha \log 2)v_2 > (2\alpha\log 3 - 7)\lfloor\frac{v_3-3}{2}\rfloor + (\alpha \log 5 -6)v_5 + C_3(\alpha). $$

Here $C_3(\alpha)$ is as follows: $C_3(\alpha) = \alpha \log 24 - 10$, when $\alpha \in [\frac{26}{\log 1439}, \frac{4}{\log 3}]$, $C_3(\alpha) = 3\alpha \log 2 -6$ when $\alpha \in (\frac{4}{\log 3},\frac{6}{\log 5}],$ and $C_3(\alpha)= \alpha \log (8/5)$ when $\alpha \in (\frac{6}{\log 5}, \frac{3}{\log 2})$.

\label{2,5,9}

\end{proposition}

\begin{proof} The proof uses the reduction scheme $[1,9]^e[T][1,5]^{v_5-1}[I][1,2]^v_2n$. Here we have $[I]$ is the initial gadget either $[0,4]$ or $[1,6][0,2]$. We have for the transition gadget $[T]$ either $[0,2]$ or $[0,2][1,5]$. Here we have $e=\lfloor (v_3-1)/2\rfloor$ (we need the -1 in $e$ since we may have used a 3 in the in our initial gadget).    Following that the reductions are valid uses logic similar to that in the previous proposition.  We use that $v_7 \geq 1$ to insure that the number we eventually reduce to is at least 1. 

\end{proof}

From Proposition \ref{2,5,9} and Proposition \ref{3then2} as well as our lower bounds on $v_2$, $v_3$, $v_5$, $v_7$ we obtain:

\begin{theorem}\label{moderateupperbound} Assume $\alpha \geq \frac{25}{6 \log 2 + 2 \log 3}$. Then for all $n >1$, we have $\cpx{n} \leq \alpha \log n$. 
\end{theorem}
\begin{proof} This proof is the same method as with Theorem \ref{weakupperbound}, but with our strengthened Proposition \ref{2,5,9}  being used also.  
\end{proof}

We may then using Theorem \ref{moderateupperbound} always assume in all subsequent results that we have $\alpha \in I_2$. \\

From Proposition \ref{2,5,9}, we obtain:

\begin{corollary} Assume that $\alpha \geq \frac{41}{\log 55296} = 3.7544  \cdots $  Assume further that $S(\alpha,n)$ fails. 
Then \begin{equation}\label{cor4eq} v_2 > 1.57093v_3 + 0.1069v_5 +1.29586.\end{equation}

\label{cor4}

\end{corollary}

\begin{lemma} Assume that $\alpha \geq \frac{41}{\log 55296} = 3.7544 \cdots  $, and that $S(\alpha,n)$ fails. Then $v_{17} >1$. 
\label{v_17=0}
\end{lemma}
The proof is in the Appendix.\\ 

% If trying to tighten further, could possibly move this lemma up slightly. Difficulty is using all available 3s in current form.  But maybe could use [1,6] in some occasions where use [1,18] now.  A lot of the do not need there ending [1,18], but the cases marked above would have issues. 

We can now give a slightly tighter version of Proposition \ref{2,5,9}
by instead making our final gadget $[F]$ one of the following:  $[1,34][0,2]$, $[2,33][1,17][0,2]$, $[1,51][0,2]$, $[1,6][1,17][0,2]$, $[1,19][1,6]$, $[0,3]$, $[2,51]$. To see that all cases are covered, let $k$ be our number right before the final gadget. Note that at least one of the first four reductions will work for any given even k depending on (k mod 8), so we may in the last three  options for our final gadget assume that $k$ is odd. Then we do each of the three others based on $k$ (mod 3). (We will use similar logic for later final gadgets but will not in general explicitly list every modulus used in a case by case breakdown.)
We then obtain:

\begin{proposition}

Assume that $S(\alpha,n)$ fails for some $\alpha \in I_2$. Assume also that $v_2 \geq 1$, $v_3 \geq 3$, $v_5 \geq 2$, $v_7 \geq 1$, $v_{11} \geq 1$, $v_{13} \geq 1$, $v_{17} \geq 1$, $v_{19} \geq 1$. Then

$$(3-\alpha \log 2)v_2 > \frac{2\alpha\log 3 - 7}{2}v_3 + (\alpha \log 5 -6)v_5 + C_4 + C_5 . $$

Here $C_4 = \alpha \log (8/15) + 3$ when $\alpha \in (\frac{26}{\log 1439}, \frac{6}{\log 5}]$, and 
$C_4 = \alpha \log (8/3) -3$ when $\alpha \in (\frac{6}{\log 5}, \frac{25}{6\log 2 + 2 \log 3})$. 

% Coming from $C_4=(3 \alpha \log 2 - 6) +\mathrm{min}(0,3 - \alpha \log 3)) + \mathrm{min}( 0, 6 - \alpha \log 5)$.

Here $C_5 =\alpha \log 51- 14 $ when $\alpha \in (\frac{26}{\log 1439}, \frac{11}{\log 17}]$, and $C_5= \alpha \log 3 - 3$ when $\alpha \in (\frac{11}{\log 17},  \frac{25}{6\log 2 + 2\log 3})$.

%$C_5$ arises from the minimum of $$\{\alpha\log  68 - 14, \alpha \log 1122 -25, \alpha\log 102 - 15, \alpha \log 204 -18,  \alpha \log 114 - 16, \alpha \log 3 -3, \alpha \log 51 -14\}$$

(Here for readability we have split $C_4$ which arises from our initial and transition gadgets and $C_5$ which arises from our final gadget.)
\label{tight2,5,9prel}
\end{proposition}

As usual, we get a corollary from this: 
%Current value relevant constants are $\alpha \log 51 -14$ and $\alpha \log(8/3) -3 $, so constant comes from $\alpha \log 136 - 17$

\begin{corollary}

Assume that $\alpha \geq \frac{41}{\log 55296} = 3.7544 \cdots \cdots$ and that $S(\alpha,n)$ fails.  Then \begin{equation}\label{corfromtight2,5,9prel equation}
v_2 > 1.57093v_3 + 0.1069v_5   +4.04879.\end{equation}
\label{corfromtight2,5,9prel}

\end{corollary}
 
We can use the above Proposition, together with the fact that  $S(\alpha,n)$ fails in our range then $v_3 \geq 4$ and $v_5 \geq 2$ to conclude that: 

\begin{lemma} If  $\alpha \geq  \frac{41}{\log 55296} = 3.7544 \cdots$, and $S(\alpha,n)$ fails, then $v_{2} \geq 11$.
\label{v_2=12}
\end{lemma} 
 
 We would like to be able to use the construction from Proposition \ref{tight2,5,9prel} to get that $v_2 \geq 12$, but that construction by itself is insufficient to reach that conclusion. We will present a similar version which has a slightly modified final gadget which is optimized near our choice of $\alpha$:
 
 \begin{proposition} Assume that $S(\alpha,n)$ fails for some $\alpha \in I_2$. Assume further that $v_2 \geq 1$, $v_3 \geq 3$, $v_5 \geq 2$, $v_7 \geq 1$, $v_{11} \geq 1$, $v_{13} \geq 1$, $v_{17} \geq 1$, and $v_{19} \geq 1$. Then

$$(3-\alpha \log 2)v_2 > \frac{2\alpha\log 3 - 7}{2}v_3 + (\alpha \log 5 -6)v_5 + C_4 + C_5 . $$

Here $C_4 = \alpha \log (8/15) + 3$ when $\alpha \in (\frac{26}{\log 1439}, \frac{6}{\log 5}]$, and 
$C_4 = \alpha \log (8/3) -3$ when $\alpha \in (\frac{6}{\log 5}, \frac{25}{6\log 2 + 2 \log 3})$. $C_6= \alpha \log 1071-25$ when $\alpha \in [\frac{26}{\log 1439}, \frac{11}{\log 17}]$, and $C_6 = \alpha \log 63 - 14$ when $\alpha \in [\frac{11}{\log 17}, \frac{26}{6\log 2 + 2\log 3}]$.
% In our current range, have $C_6 = \alpha \log 63 -14$.
\label{tight2,5,9}

 \end{proposition}
 \begin{proof} The construction is identical to that of the proof for Proposition \ref{tight2,5,9prel} but with the following changes. We extend the final gadget  $[0,3]$ to one of $[1,6][0,3]$, $[0,9]$ or $[2,21][0,3]$. Similarly, we extend $[2,51]$ to one of $[0,3][2,51]$, $[2,21][2,51]$ or $[1,6][2,51]$. We extend $[2,33][1,17][0,2]$ with either to do either $[1,4][2,33][1,17][0,2]$ or $[1,14][1,2][2,33][1,17][0,2]$. Our entire tail list then is: 
 
 $[1,34][0,2]$, $[1,4][2,33][1,17][0,2]$, $[1,14][1,2][2,33][1,17][0,2]$, $[1,51][0,2]$, $[1,6][1,17][0,2]$, $[1,19][1,6]$,  $[1,6][0,3]$, $[0,9]$, $[2,21][0,3]$,   $[0,3][2,51]$, $[2,21][2,51]$ or $[1,6][2,51]$. The worst case situations for this new final gadget then are $[2,21][0,3]$, and $[2,21][2,51]$
 
 \end{proof}
 
 We then have
 
 \begin{corollary} Assume that $\alpha \geq \frac{41}{\log 55296} = 3.7544 \cdots \cdots$ and that $S(\alpha,n)$ fails.  Then \begin{equation}\label{corfromtight2,5,9 equation}
v_2 > 1.57093v_3 + 0.1069v_5   +5.6271.\end{equation}
\label{corfromtight2,5,9}
 
 \end{corollary}

From which we immediately get:
\begin{lemma} If  $\alpha \geq \frac{41}{\log 55296} = 3.7544 \cdots \cdots$ and that $S(\alpha,n)$, then $v_2 \geq 13$.
\end{lemma}

This now allows us to prove:

\begin{lemma}  If $\alpha \geq \frac{41}{\log 55296} = 3.7544 \cdots $, and $S(\alpha,n)$ fails, then $v_7 \geq 2$.
\label{v_7=2}
\end{lemma}
\begin{proof} We may assume that   $v_2 \geq 12$, $v_3 \geq 4$, $v_5 \geq 2$, $v_7 \geq 1$, $v_{11} \geq 1$, $v_{13} \geq 1$, $v_{17} \geq 1$, and $v_{19} \geq 1$. \\ 

Case {\bf I}: $n \equiv 6$ (mod 49). We may use the reduction $[1,8]^3[6,49]n$. \\

Case {\bf II}: $n \equiv 13$ (mod 49). We may use the reduction $[1,4][1,8]^3[6,49][1,2]n$ \\

Case {\bf III}: $n \equiv 20$ (mod 49). We may use the reduction

$[1,8]^4[6,49][2,3]n$.\\

Case {\bf IV}: $n \equiv 27$ (mod 49). We may use the reduction $[F][2,15]^2[1,8]^3[6,49][1,2]^2n.$ \\  
Here $[F]$ is one of $[0,2]$, $[1,4]$ or $[1,22][1,2]$ \\ % NOTE THIS IS A DECENT TAIL IF NEED TO DO SYSTEMATIC v_7 later. 
 
 Case {\bf V:} $n \equiv 34$ (mod 49). We may use different reductions based on  $k=[1,8]^i[6,49][5,4]n$. Here $i$ is the maximum number of times we may apply $[1,8]$, noting that $i$ is at least 4, and that $k$ is not $1$ (mod 8). 
If $k \equiv 0$ (mod 2) then we may use the reduction
$[1,19][1,3][1,15][0,2][1,8]^i[6,49][5,4]n$,
 
So we may assume that $k$ is odd. If $k \equiv 1$ (mod 4), then since $k$ cannot be $1$ (mod 8), we may use

or $[1,15][1,18][1,4][1,8]^i[6,49][5,4]n$

We  may then assume that $k \equiv 3$ (mod 4).

If $k \equiv 3$ (mod 4), we may use one of \\

$[2,57][1,176][1,2][1,8]^i[6,49][5,4]n$, \\ 

$[1,114][1,88][1,2][1,8]^i[6,49][5,4]n$, \\ 

$[2,57][1,44][1,2][1,8]^i[6,49][5,4]n$, \\

or 
$[1,30][1,22][1,2][1,8]^i[6,49][5,4]n$. \\ 

Case {\bf VI:} $n \equiv 41$ (mod 49).

We may then take $[1,9][1,4][1,8]^3[6,49][2,3][1,2]n$.

\end{proof}

We now turn to bounding $v_3$ from below.

\begin{proposition} Assume that $\alpha \in I_2$ and $S(\alpha, n) $ fails. Assume further that $v_3 \geq 1$, $v_5 \geq 2$, $v_7 \geq 1$, $v_{11} \geq 1$, $v_{17} \geq 0$, $v_{19} \geq $. Assume there are integers $i_5$ and $i_7$ such that $2 \leq i_5 \leq v_5$ and $0 \leq i_7 \leq v_7$, and that $v_2 \geq 2i_7 + i_5 + 7$. Then we have:
\begin{equation} v_3 > (\alpha \log 28 -11)i_7 + (\alpha \log 10 - 8)i_5 + (\frac{\alpha \log 4- 5}{2})(v_2 - 2i_7 - i_5-4) + C_{11} +C_{12} . \label{general form for 28-10-3 with using i5 and i7 variables}
\end{equation} Here $C_{11} = \alpha \log 19 - 11$ if $\alpha \in [\frac{26}{\log 1439}, \frac{11}{\log 19}]$ or $C_{11}=0$ when $\alpha \in (\frac{11}{\log 19}, \frac{25}{6 \log 2 + 2 \log 3}]$. Here $C_{12} = \alpha \log 459 - \frac{39}{2}$.
\end{proposition}
\begin{proof}  Assume as given. We will use the reduction 
$$[F][1,28]^{i_7}[T_2][1,10]^{i_5-1}[T_1][1,4]^{e_i}[I][2,3]^{v_3}n.$$
We have $[I]$ is either $[0,3]$ or $[1,6]$. We have $e_i = \lfloor\frac{v_2 - i_5 - 2i_7 - j -4}{2}\rfloor $ where $j=1$ if we used  in $[I]$ and $j=0$ if we used $[0,3]$. We have $[T_2]$ is either $[0,3]$ $[0,3][1,4]$, or $[0,3][1,4]^2$. $[T_2]$ then is one of $[0,3][1,10][2,19]$, $[0,3][2,19]$, or $[0,3][1,10]$.
Our final gadget $[F]$ has multiple parts. We will write $[F]=[F_2][F_1]$. $[F_1]$ is either empty or if we had $[T_2] = [0,3][1,10][2,19]$  or $[T_2] = [0,3][1,10]$ and is $[2,55]$ otherwise. $[F_2]$ is either $[1,17][1,4]$, or is $[1,34][1,4]$; we can do the second whenever $v_2 - i_5 - 2i_7 - j -4$ is odd. 
Changing which of $[T_2]$ options is our worst case scenario contributes the constant term of $C_{11}$, while the rest of our gadgets contribute the $C_{12}$ term.

\end{proof}
\begin{corollary} Assume that $\alpha \geq \frac{41}{\log 55296} = 3.7544 \cdots$, and that $S(\alpha,n)$ fails. Assume there are positive integers $i_5$ and $i_7$ such that $1 \leq i_5 \leq v_5$ and $1 \leq i_7 \leq v_7$, and that $v_2 \geq 2i_7 + i_5 + 7$  Then \label{i5 and i7 cor}
\begin{equation} \label{i5 and i7 cor equation} 1.49171i_7 +0.61976i_5    +0.11694v_2 +3.54325.
\end{equation}
\end{corollary}

\begin{lemma} If $S(\alpha)$ fails for some $ \frac{41}{\log 55296} = 3.7544 \cdots$, then $v_3 \geq 9$, and $v_2 \geq 22$. Under the same assumption we also have $v_2 \geq v_{3} +11$.
\label{First consequence of general v3 form}
\end{lemma}
\begin{proof} We repeatedly apply Corollary \ref{i5 and i7 cor} with $i_5=2$, $i_7 =2$ along with \ref{corfromtight2,5,9} using that we must have integer values for $v_2$ and $v_3$. We first use to conclude that $v_3 \geq 9$. We then conclude that $v_2 \geq 20$, and so conclude that $v_3 \geq 10$, and thus $v_2 \geq 22$. We apply again Corollary \ref{i5 and i7 cor} to gain the last inequality. 
\end{proof}

%1.57093v_3 + 0.1069v_5   +5.6271

We are now in a position where we can bound $v_5$ from below. 

\begin{proposition} Assume that $\alpha \in I_2$. Assume that $v_3 \geq 1$, $v_5 \geq 1$, $v_7 \geq 1$, $v_{11} \geq 2$, $v_2 \geq v_3 +11$, and that $S(\alpha, n)$ fails. Then we have $$(9 - \alpha \log 5)v_5 \geq C_{13} +C_{14} +C_{15} + \frac{\alpha \log 16 - 9}{4}(v_2 - v_3) + (\alpha \log 6-6)v_3. $$

Here, $C_{13} = \frac{11}{4}- \alpha \log \frac{16}{5} $,  $C_{14}=\alpha \log 33-13$, and $C_{15}=\alpha \log 114 -17$. 
\label{16-6-5 prop}
\end{proposition}
\begin{proof}

  We perform the following reduction:

$$[F][1,16]^{e_2}[T][1,6]^{e_3}[I][4,5]^{v_{5}-1}.$$ Here, 

$[I]$ is the initial gadget given by one of $[4,25]$, $[4,25][1,2]$, $[4,25][2,3]$ or $[4,5][1,2][1,2]$.

We then have $e_3$ is ten as large as possible; we'll have $e_3=v_3$, except when in the initial gadget we used the reduction $[4,25][2,3]$, in which case we have $e_3=v_3-1$. 

We note that $m=[1,6]^{e_3}[I][4,5]^{v_{3}-1}$ cannot be $1$ (mod 3). We then use the transition gadget $[T]$ as either $[0,3]$ or $[2,33]$. The number after $[T]$ either way will allow repeated $[1,16]$ reduction as $-1/15$ has in the $2$-adics the repeating digits $0000100001 \cdots $. We have then that $e_3 = \lfloor\frac{v_2-v_3}{4} \rfloor$ when we are in the the first and third of our cases for our initial gadget. In our second case where use used $[I]=[4,25][1,2]$ we have $e_3 = \lfloor\frac{v_2-1-v_3}{4} \rfloor$, and in our fourth case where we have $[I]=[4,25][1,2][1,2]$ we have $e_3 = \lfloor\frac{v_2-1-v_3}{4} \rfloor$.

We then consider our final gadget $[F]$, by considering $$k=[1,16]^{e_2}[T][1,6]^{e_3}[I][4,5]^{v_{5}-1}$$

If $k \equiv 5$ (mod 6), we make $[F]=[2,57][1,2]$. If $k \equiv 0$ (mod 3), we make $[F] = [2,91][0,3] $. If $k \equiv 1$ (mod 6), we do $[2,49][1,6]$. If $k \equiv 1$ (mod 6) we do $[2,49][1,6]$. If $k \equiv 0$ (mod 3), we do $[2,91][0,3]$. We may thus assume that $k$ is even and that $k =1$ (mod 3) or $k \equiv 2$ (mod 3). If $k \equiv 0$ mod 16, we do $[0,16]$. If $k \equiv 8$ (mod 16) we do $[1,242][0,8]$. If $k \equiv 4$ (mod 8), we do $[2,121][0,4]$. If $k \equiv 2$ (mod 4), we do $[1,91][1,2][0,2]$ 

Let us now discuss $C_{10}$, $C_{11}$, and $C_{12}.
$
$C_{10}$ arises from the contribution of $[I]$, and that $e_2$ may not use every power of 2.  We have a slight advantage here in that one might think one needs to take the worst case scenario of both $e_2$ and $[T]$,  but we cannot both loses powers of 2 to $[I]$ and lose a power of 3 to $[I]$. Our worst case turns out to always be our third case, where we have a $[2,3]$ instead of a $[1,6]$. Our worst case for $e_3$, then becomes not getting to use all three 2s. We thus have
$$C_{10} = (\alpha \log 5) - 5  - \left(\frac{3}{4}\left(\alpha \log 16 -9\right) + \left(1+(\alpha \log 2 -2) \right)\right)= \frac{11}{4}- \alpha \log \frac{16}{5}. $$

$C_{11}$ arises from our transition gadget $[T]$, which given our range of $\alpha$ always has $[2,33]$ as the worst case, hence $C_{11}= \alpha \log 33-13$.

$C_{12}$ contains the contribution from our final gadget $[F]$. The worst case scenario here is always $[2,57][1,2]$, which gives $C_{12}$.

\end{proof}

%Other option for tail above  set is have $v_{31}$, $v_{23}$ and $v_{29}$, to do 
%$[1,31]^{v_{31}}[0,2]$
%$[1,31]^{v_{31}}[1,46][1,2]$, 
%$[1,31]^{v_{31}}[1,38][1,4]$
%$[1,31]^{v_{31}}[1,34][1,8]$

%May be better than above. Is definitely better when have $v_{31} \geq 2$.

% Also can extend this one if just have $v_{23} \geq 1$, by just tightening the currently weakest $[2,57][1,2]$ line. 

As usual, the above proposition has an associated corollary:

\begin{corollary} Assume that  $\alpha \geq  \frac{41}{\log 55296} = 3.7544 \cdots$, and $S(\alpha, n)$ fails. Then \begin{equation}\label{v5 general cor equation}v_5 > 0.11914v_2 + 0.12668v_3 - 0.23936 .\end{equation} \label{v5 general cor}
\end{corollary}

\begin{lemma} If  $\alpha \geq  \frac{41}{\log 55296} = 3.7544 \cdots$, and $S(\alpha, n)$ fails, then $v_2 \geq 27$, $v_3 \geq 13$ and $v_2 \geq 27$.

\end{lemma}
\begin{proof} We use the same tactic as with Lemma \ref{First consequence of general v3 form}, but now iterate also the inequality from Equation\ref{v5 general cor equation}. We first conclude that $v_5 \geq 4$, and so we may take $i_5=4$, and $i_7 =2$ (justified by our increase in $v_2$). We then get that $v_3 \geq 12$, and so $v_2 \geq 25$, and continue this way process, obtaining that
$v_5 \geq 5$, $v_2 \geq 26$, $v_3 \geq 13$. and then that $v_2 \geq 27$.

\end{proof}

We now present another lower bound for $v_2$. This one essentially extends our earlier bounds of Propositions \ref{tight2,5,9prel} and Proposition \ref{tight2,5,9prel} where instead of a complicated final gadget, we have another transition gadget and then repeated reduction of $[1,17]$. After that, we a new, somewhat simple, final gadget, which depends on how much 5 reduction we used earlier.  

\begin{proposition}

\label{2,3,5,9,17}
Assume that $\alpha \in I_2$ and $S(\alpha,n)$ fails. Assume further that  $v_2 \geq 1$, $v_3 \geq 4$, $v_5 \geq 2$, $v_{11} \geq 1$, $v_{13} \geq 1$, and $v_{17} \geq
1$. Then
$$(3-\alpha \log 2)v_2 > \frac{\alpha \log 9 - 7}{2} v_3 + (\alpha \log 5 - 6 )v_5
 + (\alpha \log 17 - 10)v_{17} + C_{10}.$$
Here $C_{10}= \alpha \log \frac{40678}{243} - 18.5$ when $\alpha \in (\frac{26}{\log 1439}, \frac{14}{\log 49}]$, $C_{10} = \alpha \log \frac{4160}{243} -10.5 $ when $\alpha \in [\frac{14}{\log 49}, \frac{6}{\log 5}]$, and $C_{10}= \alpha \log \frac{832}{243} - 4.5$ when $\alpha \in  (\frac{6}{\log 5}, \frac{25}{6\log 2+2\log 3}]$.

\end{proposition}

\begin{proof}
We use the following reduction 
$$[F][1,17]^{v_{17}}[T_2][1,9]^e[T_1][1,5]^{v_5-1}[I][1,2]^{v_2}n$$

Where the indicated gadgets are defined as follows. 
Our initial gadget, $I$, occurs after burning as many $[1,2]$ possible. We then have an initial gadget of either $[0,4]$ or $[1,6]$. We then reduce by $[1,5]^{v_5-1}$. $T_1$ then is either $[1,10]$ or $[0,2]$ so we have effectively divided by another 2. We then repeatedly  reduce by $[1,9]$ as many times as we can except for the last one which we then do either $[0,2]$ or $[1,18]$ for $[T_2]$, and then repeatedly reduce by $[1,17]$, using all of them.  For $F$ our final gadget we either apply $[0,2]$ or $[1,49][1,2]$ (depending on whether we have an even or odd number). We then apply either $[0,4]$ or $[1,66][0,2]$ and then use either a $[1,65]$ if we did not use our final 5 in $T_1$  earlier and using just $[1,13]$ if we used our final 5. Note that in terms of our $p$-adic expansions this result amounts to $-1/4$ being $1$ repeating in the $5$-adics, $-1/8$ is $01$ repeating in the $3$-adics, and $-1/16$ is $1$ repeating in the $17$-adics. 
\end{proof}

As before we plug in our desired value of $\alpha$ into the above proposition to obtain our corollary:

% Term here comes from last one with C_10 in above 6/log 5 range. 

\begin{corollary} \label{2,3,5,9,17cor} If  $\alpha \geq  \frac{41}{\log 55296} = 3.7544 \cdots$ , and $S(\alpha,n)$ fails then:
\begin{equation}v_2 > 1.57093v_3 + 0.1069v_5+ 1.60217v_{17} + 0.30388. \label{2,3,5,9,17cor equation} \end{equation}

\end{corollary}

We now turn to establishing that we may always take $i_5= v_5$ and $i_7=v_7$ for our given range of $\alpha$. This is the same as showing that $v_2 \geq 2v_7 + v_5 + 7$. 

Our next proposition will help us do that. 

\begin{proposition}
\label{2-3-7-13-25-49} Assume that $\alpha \in I_2$ and that $S(\alpha,n)$ fails. Assume further that If $v_2 \geq 1$, $v_3 \geq 1$, $v_5 \geq 3$, $v_7 \geq 3$, $v_{11} \geq 1$, and $v_{13} \geq 1$. Then
have $$(3-\alpha \log 2)v_2 > C_{15} + \beta v_7 + (\alpha \log 25 -11)\frac{v_5-1}{2}  + \gamma v_{13} + (\alpha\log 3-4)v_3. $$
Here we have $\beta = \frac{\alpha \log 49 - 13}{2}$, $\gamma =\alpha \log 13 -9$. $C_{15}$ is given by the following: If $\alpha \leq \frac{7}{\log 7}$, $C_{15}= \alpha \log \frac{1824}{25} -15 $. If $\alpha \in (\frac{7}{\log 7},\frac{25}{6 \log 2 + 2\log 3}] $, then $C_{15} = \alpha \log \frac{48}{25} -4$.
\end{proposition}
\begin{proof} The reduction path here is to use 

$$[F][1,49]^{e_7}[T_3][1,25]^{e_5}[T_2][1,13]^{v_{13}}[T_1][1,3]^{v_3}[1,4][1,2]^{v_2-1}n. $$

Here, $[T_1]$ is the transition gadget $[0,3]$ or $[2,15]$ (Here we are using that
$[1,3]^{v_3}[1,4][1,2]^{v_2-1}n$ cannot be $1$ (mod 3)), followed by either $[1,14]$ or $[0,2]$ .We now have a number that behaves like $-1/12$ $p$-adically for $p=5$, $7$ and $13$. After our reductions of $[1,13] $, we use $[T_2]$  which is $[0,2]$ if we have reduced to an  even number and
$[1,19][1,2]$ if have reduced to an odd number. We set $e_5 = \lfloor{\frac{v_5}{2}-1}\rfloor$. And then for $[T_3]$ use either $[0,2]$ or $[1,50]$ depending on whether our number up to this point is even or odd. We then reduce repeatedly  by as many $[1,49]$, so $e_7 = \lfloor \frac{v_7}{2}\rfloor$. Then $[F]$ is then $[1,7]$ if $v_7$ was odd, since we have one more $7$ we can use it to reduce further. 
\begin{comment}
Constant becomes $\alpha \log 2 -2$ from first 2. Then   part from $[2,15]$ - $\alpha \log 15 - 10$. For the rest of $T_1$, cutoff point is $\alpha =\frac{7}{\log 7}$, to the left of which $[1,14]$ is the worst case contributing $\alpha \log 14 - 9$, and to the right of which the worst case is $[0,2]$.
Then, for $[T_2]$, contribution is if $\alpha \leq \frac{11}{\log 19}$, $\alpha 38 - 13$, and is $\alpha \log 2 -2$, otherwise. Note that $\frac{7}{\log 7}$ cutoff is smaller than $\frac{11}{\log 19}$

For $[T_3]$, in $I_2$ range is always $-\frac{3}{2}(\alpha \log 25 - 11) + \alpha \log 2 - 2 =$

Then, for final $7$ issue may have a cost of 1.

In all cases  adding in $-\frac{3}{2}(\alpha \log 25 - 11) + \alpha \log 2 - 2$ , $-1$, and $\alpha \log 15 -10$, $\alpha \log 2 -2 $, and possibly lose an additional 7 , so 

All terms term then is $\alpha \log \frac{12}{25} + \frac{1}{2}$ - $\frac{\alpha  \log 49 - 13}{2} = \alpha \log \frac{12}{175}-7$ = 

(from last seven issue).

\end{comment}

\end{proof}

\begin{corollary} If   $\alpha \geq  \frac{41}{\log 55296} = 3.7544 \cdots$ and $S(\alpha,n)$ fails then  \begin{equation}\label{v2 big compared to v5 and 2v7 equation}v_2 >  0.31349v_3 + 1.36434v_5 + 1.5841v_{13} + 2.02641v_7 -5.2647.  \end{equation}

\label{2-3-7-13-25-49cor}
\end{corollary}

The constant in the above corollary is unfortunately too negative; in order to do our general reduction to lower bound $v_3$, we will need to know that $v_2 \geq v_5 + 2v_7 +7$. There are two obvious approaches to improving this bound. The first is to throw in a final gadget. Unfortunately, the obvious final gadget would involve reductions using $p=97$. If we had that $v_{97} \geq 1$, we could use the final gadget $[2,97]$, which would help somewhat. Proving that sort of Lemma would be difficult without at least having much higher $v_p$ for many smaller $p$. The other method is to bump up $v_3$, $v_5$, $v_7$, and $v_{13}$ high enough that they give us the inequality we need. The inequality in Corollary \ref{2-3-7-13-25-49cor} above is not strong enough by itself due to the large negative constant on the right-hand side.Note that  Corollary \ref{2-3-7-13-25-49cor} is strong enough to get that $v_2 \geq v_7 +9$ which we will use below. 

\begin{proposition}\label{v13 lower bound} Assume that $\alpha \in I_2$ and $S(\alpha,n)$ fails. Assume further that $v_2 \geq v_7 +8$. Assume that $v_3 \geq 4$, $v_5 \geq 1$, $v_{11} \geq 1$. $v_{13} \geq 1$, $v_{17} \geq 1$. Then 
$$\frac{(\alpha \log 27 - 11)}{3}v_3 + (\alpha \log 14 - 9)v_7 + (\alpha \log 80 -15)m_1 +C_{21} C_{22} \leq (15-\alpha \log 13)v_{13}.$$ Here $C_{21} = \alpha \log \frac{832}{3} - 20$, and $m_1 =\min(\lfloor \frac{v_2-v_7-6}{3}\rfloor, v_5 )$.  
\end{proposition}
\begin{proof} We use the reduction $$([1,20][1,4])^(m_1)[2,27]^{e_3}[1,14]^{v_7}[I][12,13]^{v_{13}}n.$$ Here $[I]$ is given by one of  $[0,13]$, $[1,26]$, $[2,39]$, $[1,26][1,2]$, $[2,39][1,2]^4$, $[2,39][1,2]$, $[2,39][1,2]^6$, $[1,26][1,2]^2$, $[2,39][2,3]$, $[2,39][1,2]^5$, $[2,39][1,2]^3$, $[2,3][1,2]^2$.
$e_3 = \lfloor\frac{v_3-j}{3}\rfloor$ where $j$ is the number of 3s burned in $[I]$.  Our worst case scenarios occur either when $[12,13]^{v_{13}}n \equiv 5$ (mod 13), and we have to use $[2,39][1,2]^6$, or when $[2,39][2,3]$, and where we then lose two 3s in $v_3$ which we cannot use; this rise to $C_{21}$. 
\end{proof}

\begin{corollary} If  $\alpha \geq \frac{41}{\log 55296} = 3.7544 \cdots$ and $S(\alpha,n)$ fails then \begin{equation}\label{v13 lower bound cor equation}v_{13} > 0.08528v_3 + 0.1691v_7 +0.27038m_1 + 0.20845. \end{equation}
\label{v13 lower bound cor} As before, $m_1 =\min(\lfloor \frac{v_2-v_7-6}{3}\rfloor, v_5) $.
\end{corollary}

We obtain from this:

\begin{lemma} If  $\alpha \geq \frac{41}{\log 55296} = 3.7544 \cdots$ and $S(\alpha,n)$ fails then $v_{13} \geq 3$, and $m_1 \geq 2$. We also have $v_2 \geq 2v_5 + v_7 + 8$.
\end{lemma}
\begin{proof} 
We may have that $m_1 \geq 1$ by earlier remarks. We then conclude that $v_{13} \geq 2$.  This allows us to get that $m_1 \geq 2$ This allows us to use Corollary \ref{2-3-7-13-25-49cor} get that $v_2 \geq 2v_5 + v_7 + 8$. \end{proof}
%m2 >=5

\begin{proposition} Assume that $\alpha \in I_2$, and that $S(\alpha,n)$ fails. Assume that $v_2 \geq v_3 + 7$. Assume also that $v_3 \geq 5$, $v_5 \geq 1$, $v_7 \geq 1$, $v_{11} \geq 1$, $v_{13} \geq 1$, $v_{17} \geq 1$, and $v_{19} \geq 1$.  Then $$ (17 -\alpha \log 17))v_{17} > (\alpha \log 18 - 9)\frac{v_3}{2} + (\alpha \log 35 - 13) )m_3 + C_{24} .$$ Here $C_{24} = \alpha \log (\frac{17(32)}{3\sqrt{2}})) - \frac{47}{2}$. Here $m_3 = \min(v_5,v_7)$. 
\end{proposition}
\begin{proof} We use the reduction $[2,35]^{m_3}[1,18]^{e_3}[I][16,17]^{v_{17}}$ Here $[I]$ is the initial gadget given by one of $[0,17]$, $[1,34]$, $[2,51]$, $[1,34][1,2]$, $[2,51][1,2]^6$, $[2,51][1,2]$, $[2,51][1,2]^3$, $[1,34][1,2]^2$
$[2,51][2,3]$, $[2,51][1,2]^7$, $[2,3][1,2]^2$, $[2,51][2,3]^3$ $[2,51][1,2]^4$
$[2,51][2,3][1,2]^6$, $[1,34][1,2]^3$. We have then $e_3 = \lfloor{v_3-j} \rfloor$ where $j$ is the number of 3s used in [I]. 
The worst case is when we use $[2,51][2,3][1,2]^6$, and then lose a 3 in to the floor in the definition of $e$.
\end{proof}

\begin{corollary} Assume $\alpha \geq \frac{41}{\log 55296} = 3.7544 \cdots$  and $S(\alpha,n)$ fails. Then \begin{equation}\label{v17 cor eq}v_{17} > 0.1455v_3 + 0.054735m_3  -0.82932. \end{equation} Here $m_3 = \min(v_5,v_7).$
\end{corollary}

We note that since $m_3 \geq 2$, we immediately get from the above that we may assume that $v_{17} \geq 2.$

\begin{proposition}\label{7-8} Assume that $\alpha \in I_2$ and $S(\alpha,n)$ fails. Assume further $v_2 \geq 11$, $v_3 \geq 2$, $v_5 \geq 2$, $v_7 \geq 1$, and $v_{11} \geq 1$. Then
$$(11- \alpha \log 7) v_7 > \frac{3\alpha \log 2 -7}{3}v_2 + \frac{(\alpha \log 15 -9)}{2}m_2 + C_{18} + C_{19}.$$
Here $C_{18} = \alpha \log 21 - \frac{40}{3}$, and $C_{19} = \alpha \log 44 -14$, when $\alpha \in [\frac{26}{\log 1439}, \frac{12}{\log 22}]$, and $C_{20} = \alpha \log 2 - 2$, when $\alpha \in (\frac{12}{\log 22}, \frac{26}{6 \log 2 + 2 \log 3})$.  Here $m_2 = \min(v_3 - 1,v_5)$

\end{proposition}
\begin{proof} Here we use the path $[F][1,15]^{m_2}[T][1,8]^{e_2}[I][6,7]^{v_7}n$. Here, $[I]$ is the gadget doing one of the following depending on $[1,8]^{e_2}[I][6,7]^{v_7}n$ (mod 7): $[0,7]$ $[1,14]$, $[2,21]$, $[1,14][1,2]$, $[2,21][1,2][1,2]]$ or $[2,21][1,2]$. We then set $e_2 =\lfloor\frac{v_2-j}{3} \rfloor$ where $j=1$ when our transition gadget used $[1,14]$, or $[2,21][1,2]$, $j=2$ when our transition gadget used $[1,14][1,2]$, and $j=0$ otherwise. Essentially, $j$ counts how many 2s we had to burn in our initial  gadget. Our $C_{18}$ comes from all aspects but our transition gadget $[T]$,  We then use a transition gadget $[T]$  of one of $[0,2]$, $[1,22][1,2]$ or $[1,18][1,4]$, which has either worst case $[0,2]$, $[1,22][1,2]$ which contributes our $C_{19}$ term.

\end{proof}

\begin{corollary}\label{7-8cor} If  $\alpha \geq  \frac{41}{\log 55296} = 3.7544  \cdots$ and $S(\alpha,n)$ fails then: \begin{equation}v_7 > 0.054619v_2 + 0.315942m_4 -0.45895.\label{v7 corollary equation corollary}\end{equation} \label{v7 corollary}
\end{corollary}

We then proceed iterating our inequalities as before and we obtain:

\begin{lemma} If  $\alpha \geq  \frac{41}{\log 55296} = 3.7544  \cdots$ and $S(\alpha,n)$ fails then: $v_2 \geq 35$, $v_3 \geq 18$, $v_5 \geq 7$, $v_7 \geq 4$, $v_{13} \geq 5$, $v_{17} \geq 3$.
\end{lemma}
\begin{proof} We iterate as before to obtain all the inequalities above but with $v_{13} \geq 3$. We then may use our inequalities to conclude that $v_2 -v_7 \geq 27$, and so we have that $m_1 \geq 7$. This allows us to now conclude that $v_{13} \geq 5$.
\end{proof}

 \begin{proposition} \label{2,3,5,9,17,65}
Assume that $\alpha \in I_2$ and $S(\alpha,n)$ fails. Assume further that  $v_2 \geq 1$, $v_3 \geq 5$, $v_5 \geq 3$, $v_{11} \geq 1$, $v_{13} \geq 1$, and $v_{17} \geq
1$. Then
\begin{equation} 
\begin{split}
(3-\alpha \log 2)v_2 & > \frac{\alpha \log 9 - 7}{2}(v_3 - m_6) + (\alpha \log 5 - 6 )(v_5-m_5) + (\alpha \log 17 - 10)v_{17} \\
 &  + (\alpha \log 33 - 12)m_6 +  (\alpha \log 65 - 14)m_5+C_{20}.  
\end{split}
\end{equation}

%$$(3-\alpha \log 2)v_2 > \frac{\alpha \log 9 - 7}{2}(v_3 - m_6) + (\alpha \log 5 - 6 )(v_5-m_5)
% + (\alpha \log 17 - 10)v_{17} + (\alpha \log 33 - 12)m_6 +  (\alpha \log 65 - 14)m_5+C_{20}.$$
Here  $C_{20}=\log \frac{16}{\log 1485} + 17$. We have $m_5 = \min(v_5-1, v_{13})$, and $m_6 = \min(v_{11}, v_{3}-4)$.
 \end{proposition}
\begin{proof}We use a reduction similar to that in Proposition \ref{2,3,5,9,17}. We 
We use the following reduction 
$$[1,65]^{m_5}[T_4][1,33]^{m_6}[T_3][1,17]^{v_{17-1}}[T_2][1,9]^e[T_1][1,5]^{v_5-1}[I][1,2]^{v_2}n$$

Where the indicated gadgets are defined as follows. 
Our initial gadget, $I$, occurs after burning as many $[1,2]$ possible. We then have an initial gadget of either $[0,4]$ or $[1,6]$. We then reduce by $[1,5]^{v_5-m_5-1}$. $T_1$ then is either $[1,10]$ or $[0,2]$ so we have effectively divided by another 2. We then repeatedly  reduce by $[1,9]$ as many times as we can except for the last one which we then do either $[0,2]$ or $[1,18]$ for $[T_2]$, and then repeatedly reduce by $[1,17]$. $[T_3]$ then is either $[0,2]$ or $[1,49][1,2]$. We have then $[T_4]$ is either $[0,2]$ or $[1,66]$.  Our worst case scenario then is when we use $[I]=[1,6]$ and every other gadget is just $[0,2]$. This gives rise to our constant term. 
\end{proof}

\begin{corollary}  If  $\alpha \geq  \frac{41}{\log 55296} = 3.7544 \cdots$ , and $S(\alpha,n)$ fails then:

\begin{equation} 
\begin{split}v_2 &> 1.57093v_3 + 0.1069v_5+ 1.60217v_{17} \\
  &+ 4.09891m_5 + 1.26426m_6 - 0.02443.
\end{split}
\end{equation}

\end{corollary}

We note that we have from earlier that $m_5 \geq 5$ and that $m_6 \geq 1$.\\

We are now ready to prove our main theorem:
\begin{proof} We iterate as before using that none of our $v_p$ may be an integer. From repeated iteration we get that $v_2 \geq 65$, $v_3 \geq 28$, $v_5 \geq 11$, $v_7 \geq 6$.  We now improve our bounds on $m_5$. Our current set of inequalities implies that $v_5 \geq 0.25v_{13} +3$, so we may take $m_5 \geq  0.25v_{13} +2$, we may repeat this process to conclude that $v_{13} \geq 0.1v_5 + 3$, $v_5 \geq 0.3v_{13} + 8$, $v_5 \geq 0.4v_{13} + 13 $, $v_5 \geq 0.5v_{13} + 20 $, $v_{13} \geq 0.11v_5 + 5$. We continue iterating this way, to get that $v_5 \geq v_{13}$ and conclude that therefore $m_5 = v_{13}$. We may then obtain that $v_3 \geq v_5 + 4$, and thus may take $m_2 = v_5$. But the resulting system has no solutions. 
\end{proof}

There are multiple possible avenues for tightening this result. At present, the primary obstacle is to have stronger definite results. It may be possible to improve on those results using an automated approach. Unfortunately, even if one has better definite results, it is likely that one will need to adjust some of the Propositions in addition to including more similar reducing paths using other primes. This adjusting would be necessary for two reasons. First, our final gadgets as written are optimized for our choice of $\frac{41}{\log 55296}$ and so one would wish to modify the gadgets. Second, one would likely need additional reduction path inequalities to obtain a contradiction. 

\section{The limits of the method}

Improving the main result beyond $\frac{41}{\log 55296}$ may be possible. The primary obstacle appears to be obtaining tighter definite results. We can easily construct additional similar inequalities or strengthen our current inequalities as long the $v_p$ are large. For example, for the reduction in Proposition \ref{16-6-5 prop}, we can extend it so that we then do repeated $[1,31]$ reduction at the end.  It may be possible to engage in an automated method of proving definite results. If one ignores definite results, one can get the following result using the same methods: 

\begin{theorem}
\label{alpha2.5ln2} There exists a finite set of primes $P$, such that for all $p$ in $P$,  $v_p$ is sufficiently large then $S(\alpha,n)$ holds for $\alpha \geq \frac{5}{2 \log 2}$. 
\end{theorem}

Note that $\frac{5}{2 \log 2} = 3.606 \cdots$ so the above result is substantially stronger as an induction step in the cases where it is valid. Unfortunately, making the set of primes small enough for this result to be useful appears difficult. There appears to be some tradeoff between having this set small and having the $v_p$ themselves not be too large. In this section, we will not attempt to explicitly construct such a set $P$ nor determine what sufficiently large is, but prove the non-explicit form above. \\

Let $p_1,p_2\cdots p_k$ and $q_1, q_2, \cdots q_{\ell}$ be primes (not necessarily distinct).  We will write $$f(v_{p_1},v_{p_2},\cdots v_{p_k}) \ldima g(v_{q_1},v_{q_2}, \cdots v_{q_{\ell}})$$ to mean that there is a constant $C$ and a finite list of primes $r_1, r_2, r_3 \cdots r_m$ such that if $S(\alpha,n)$ fails and $v_{r_1}, v_{r_2}, v_{r_3} \cdots v_{r_m}$ are sufficiently large then we have $$f(v_{p_1},v_{p_2},\cdots v_{p_k}) \leq g(v_{q_1},v_{q_2}, \cdots v_{q_{\ell}}) +C.$$  In general, the $r_i$ may include the $p_i$ and $q_i$ and most of the time well. Since in this section we will be interested only in proving results involving $\ldima$ rather than $\leq$, we will not need to pay careful attention to the specifics of the initial gadgets and transition gadgets. We will also be able to completely ignore final gadgets since all they do is improve the constant term in our linear inequalities. Note that if a system of linear inequalities with $\ldima$ is inconsistent when we replace the diamonds with less than signs, then the system will be inconsistent for sufficiently large values of the relevant variables.\\

We will define $\rdima$ as the same but with $\geq$ replacing the $\leq$ in the definition. 

We also need to discuss briefly what sorts of reductions are possible in general. This discussion was unnecessary in the previous section because we only needed to focus on making our reductions as efficient as possible. In general, if we have a prime $p$ we can burn $p$ by using the inefficient reduction $[p-1,p]$ repeatedly and then use an initial gadget to get a number that behaves like $-1/p$ with respect to other primes. This allows us to reductions of the form $$[1,p+1]^m[I][p-1,p]^{v_p-1}n$$ where $m$ is a function of the $v_q$ over the $q$ which divide $p+1$, $[I]$ is an initial gadget. In general, when we have a reduction behaving $-1/t$ we can then transition to $-1/(qt)$ where $q$ is either a prime which divides $t$ or is a prime which divides $t+1$ if we have exhausted as much $[1,t+1]$ until we have no more $q$s remaining, although this may require coordination with the initial gadget. Almost all our general indefinite reductions fit the above pattern. Proposition \ref{2-3-7-13-25-49} is a good example of these techniques. We do however need to note that if we have something behaving like $-1/t$, then we may have a other options as well. For example, when we have something acting as $-1/t$, we can instead of doing direct $[1,t+1]$ reduction we can also reduce using a base reduction that is more complicated, such as using the base 2 reduction corresponding to $-1/t$. In general, we can guarantee the existence of transition gadgets as appropriate  by using Dirichlet's theorem on primes in arithmetic progressions. However,  gadgets directly constructed this way are generally inefficient and contribute many primes in $S$.        

\begin{proposition}\label{4609-2305-1153-577-289-193-97-49-25-9-2prop} For $\alpha \geq \frac{5}{2 \log 2}$ if $S(\alpha,n)$ fails then we have:   

 $$Jv_2 \rdima  A\min(v_{11},v_{419})  +   Bv_{1153} +Cv_{193} +  + Dv_{577} + Ev_{17} + Fv_{97}$$ $$ +Gv_7 +Hv_5  + Iv_3. $$ Here $A=\alpha \log 4609 - 26$,  $B=\alpha \log 1153 - 22$, $C= \alpha \log 193 - 18$,  $D=\alpha \log 577 -20$, $E={\alpha \log 17 - 9}$, $F=\alpha \log 97 - 15$, $G=\frac{\alpha \log 49 -13}{2}$, $H= \frac{\alpha \log 25 - 11}{2}$, $I=\frac{\alpha \log 9 - 7}{2}$ and $J=3-\alpha \log 2$.
% Removed 461 term
\begin{comment}
Have  $G=(\alpha \log 289 - 18)/2$ and simplifies.
\end{comment}

\end{proposition}
\begin{proof}We will use the reduction $$k_1=[T_1][1,9]^{v_3-3}[I_1][[1,2]^{v_2-2}[n]$$ where $[I_1]$ is an initial gadget and $T_1$ is a transition gadget which correspond to first having $-1/8$ mod $3^{v_3}$, and then having $-1/24$ in the later moduli for our other primes (because we get an extra division by 3). We then further reduce $k_1$ using the reduction $$k_2=[T_4][1,97]^{v_{97}-2}[T_3][1,49]^{v_7-1}[T_2][1,25]^{\frac{v_5}{2}-2}k_1.$$ Here $T_2$ is a transition gadget which allows us to divide by an extra power of 2; this gadget is simple: just either $[0,2][1,25]$ or $[0,2]$ depending on the parity of $[1,25]^{\frac{v_5}{2}-2}k_1.$
The transition gadget $[T_3]$ similar to $T_2$, but instead is either $[0,2][1,49]$ or $[0,2]$. The transition gadget $[T_4]$ instead allows us to do one of $[0,3]$, $[0,3][1,97]$, or $[0,3][1,97]^2$ depending on  $$[1,97]^{v_{97}-2}[T_3][1,49]^{v_7-1}[T_2][1,25]^{\frac{v_5}{2}-2}k_1 {\mathrm{ mod }}3.$$ We then reduce $k_2$ using $$k_3=[T_5][1,289]^{{\frac{v_{17}}{2}}-2}[2,193]^{v_{193}}k_2.$$
Here $[T_5]$ is either $[0,2][1,289]$ or $[0,2]$. Note that we use at this stage that $\cpx{289}=17$ which comes from writing $289=1+288$ rather than the $289=17^2$. We then reduce $k_3$ further using the reduction $$[1,4609]^{\min(v_{11},v_{419})}[T_7][1,1153]^{v_{1153}-1}[T_6][1,577]^{v_{577}-1}k_3.$$ Here $[T_6],$ and $[T_7]$ are similar reductions. 
% Removed [1,2305]^{\min(v_5,v_{461})-1}
\end{proof}

%We will not for this section focus on the details of the transition gadgets as much as we have in the above proof since their exact details are unimportant for this section. 

We have an immediate corollary:
\begin{corollary}\label{4609-2305-577-289-193-97-49-25-9-2cor} 
For $\alpha \geq \frac{5}{2 \log 2}$ if $S(\alpha,n)$ fails then we have:

$$v_2 \rdima 8.85119\min(v_{11},v_{419})   +  6.855883988v_{1153} $$ $$ +4.5112v_{577} + 2.7622513v_{17}+  1.96228518v_{193}$$
 $$+ 2.99956421v_{97} + 1.22821404v_7  + 0.75563395v_5+ 0.9248125036v_3  .$$
% FT Line 5
% If needed include + 7.0971212\min(v_5,v_{461}) 
\end{corollary}

Similarly we have:

\begin{proposition}For $\alpha \geq \frac{5}{2 \log 2}$ if $S(\alpha,n)$ fails then we have:

\label{4609-2305-1153-577-289-193-97-49-25-13-4-3prop}
 $$Jv_3 \rdima A\min(v_{11},v_{419})    + Bv_{1153} +Cv_{193} +  + Dv_{577} + Ev_{17} + Fv_{97}$$ $$ +Gv_7 +Hv_5 + Iv_{13} = .$$ Here $A=\alpha \log 4609 - 26$, $B=\alpha \log 1153 - 22$, $C= \alpha \log 193 - 18$,  $D=\alpha \log 577 -20$, $E={\alpha \log 17 - 9}$, $F=\alpha \log 97 - 15$, $G=\frac{\alpha \log 49 -13}{2}$, $H= \frac{\alpha \log 25 - 11}{2}$, $I=\alpha \log 13 - 9$ and $J=5-\alpha \log 3$.
%\min(v_5, v_{461}) coef is  $\alpha \log 2305 - 24$, replace Hv_5 with H((v_5 - \min(v_5,v_{461}))
\end{proposition}
\begin{proof}
The essential reduction is very similar to the that in the previous proposition but we instead change $k_1$ slightly so that we have $k_1=[T_{1b}][1,13]^{v_{13}}[T_{1a}][1,4]^{\frac{v_2}{2}-2}[I_2][2,3]^{v_3}$ and then continue from there. The new transition gadgets needed are straightforward.
Note that the repeated $[1,4]$ reductions do not form a term on the right-hand side when $\alpha$ hits our lowest desired value, but if one had a slightly larger value of $\alpha$ one would get a $v_2$ term on the right-hand side of our inequality. 
\end{proof} 

We obtain then the next corollary:

\begin{corollary}
If $\alpha \geq \frac{5}{2 \log 2}$ and $S(\alpha,n)$ fails then we have:
\label{4609-2305-1153-577-289-193-97-49-25-13-4-3cor}

$$v_3 \rdima 4.2652484\min(v_{11},v_{419}) +   3.3037419v_{1153} + 2.824871272v_{577}$$ $$ + 1.174503128v_{17} +0.9455941v_{193} + 1.4454425v_{97} +0.49960527v_7+0.293776v_5 +0.242v_{13}  .$$
% FT Line 7
\end{corollary}
% If include min(5,461) term is + 3.490341\min(v_5,v_{461}) 

\begin{proposition}
If $\alpha \geq \frac{5}{2 \log 2}$ and $S(\alpha,n)$ fails then we have:
\label{1793-449-225-113-29-8-7prop}$$(11- \alpha \log 7)v_7 \rdima  \frac{\alpha \log 225 - 17}{2}\min(v_3,v_5)+(\alpha \log 29 - 12)v_{29} + \frac{\alpha \log 8 - 7}{3}v_2  .$$
\end{proposition}

% Can add (\alpha \log 1793 - 24)\min(v_{11},v_{163})  + (\alpha \log 449 -20)v_{449} . 113 term also an option here 

\begin{proof}
We use the reduction $$k_1=[1,29]^{v_{29}-1}[T_1][1,8]^{\frac{v_2}{3}-1}[I][6,7]^{v_7-2}.$$ Here the required transition and initial gadgets are straightforward. We then further reduce $k_1$ using the reduction $$[1,225]^{\min(v_3,v_5)/2-C}[T_2]k_1.$$ Again, the transition gadgets are straightforward.
%[1,1793]^{\min(v_{11},v_{163})}[T_5][1,449]^{v_{449}-1}[T_4]
\end{proof}

We have then the following corollary:

\begin{corollary}
If $\alpha \geq \frac{5}{2 \log 2}$ and $S(\alpha,n)$ fails then we have:
\label{1793-449-225-113-29-8-7cor}

$$ v_7 \rdima 0.31826965\min(v_3,v_5)+ 0.0364054v_{29} +  0.041859v_2 .$$
% FT Line 9
\end{corollary}
%.7585869\min(v_{11},v_{163}) + 0.508946v_{449} + 
This system is now sufficient to conclude that if $\alpha \geq \frac{5}{2 \log 2}$ and $S(\alpha,n)$ fails then $v_2 \rdima .9651v_3$. We therefore have:

\begin{proposition}If $\alpha \geq \frac{5}{2 \log 2}$ and $S(\alpha,n)$ fails then we have:

\label{6-5prop}$$ (9 - \alpha \log 5)v_5 \rdima 0.965(\alpha \log 6-6)v_3  .$$

\end{proposition}
\begin{proof} This follows from the reduction $$[1,6]^{0.965v_3}[I][4,5]^{v_5-1}n.$$ We need the $0.965$ above because we do not yet have a guarantee that $v_2$ is at least as large as $v_3$ and so must use the observation from earlier. We use $0.965$ here rather than $0.9651$ because the most straightforward versions of $[I]$ will expend some $[1,2]$ reductions (although they can strictly speaking be replaced by other  reductions).   

\end{proof}

We have then as a corollary:

\begin{corollary}

\label{6-5cor}If $\alpha \geq \frac{5}{2 \log 2}$ and $S(\alpha,n)$ fails then we have:

$$v_5 \rdima 0.139654v_3 .$$
% FT Line 11
\end{corollary}

This system is now sufficient to conclude that if $\alpha \geq \frac{5}{2 \log 2}$ and $S(\alpha,n)$ fails then $v_2 \rdima 1.05v_3$ which allows us to obtain a more substantial lower bound for $v_5$:

\begin{proposition}If $\alpha \geq \frac{5}{2 \log 2}$ and $S(\alpha,n)$ fails then we have:

\label{7681-1921-961-241-121-61-16-6-5prop}$$ (9 - \alpha \log 5)v_5 \rdima  (\alpha\log 31 - 11)v_{31}  + (\alpha \log 11 - 8)v_{11} +\frac{\alpha \log 16 - 9}{4}(v_2 - v_3) + (\alpha \log 6-6)v_3 .$$

\end{proposition}
%(\alpha \log 7681 - 28)v_{7681}+   (\alpha \log 1921 - 24)\min(v_{17},v_{113})
\begin{proof}
We use the reduction $$k_1=[1,16]^{v_2/4-2}[T_1][1,6]^{v_3-C}][I][4,5]^{v_5-1}n$$ and then reduce $k_1$ further by 
$$[1,961]^{v_{31}/2-1}[T_4][1,121]^{v_{11}/2-1}[T_2]k_1.$$ 

%We then reduce $k_2$ further by $[1,7681]^{v_{7681}-1}[T_5]k_2= [1,1921]^{\min(v_{17},v_{113})-1}[T_5]k_2.$+ (\alpha \log 61 -14)v_{61}.[1,61]^{v_{61}-1}[T_2]

\end{proof}

\begin{comment}
Could add 481=13*37 term above if need it later. 
\end{comment}

\begin{corollary}

\label{7681-1921-961-481-241-121-61-16-6-5cor}If $\alpha \geq \frac{5}{2 \log 2}$ and $S(\alpha,n)$ fails then we have:

$$ v_5 \rdima 0.433619038v_{31}+  0.2029867158v_{11} + .0782428591v_2 + 0.066477v_3 .$$
% FT Line 13
\end{corollary}
%1.33566697v_{7681}  + 1.02313633406 \min(v_{17},v_{113}) + 0.258778349v_{61}
We have a similar other path for $v_5$ with a slightly different reduction:

\begin{proposition}

\label{7681-1921-961-481-161-81-41-long2-5prop}If $\alpha \geq \frac{5}{2 \log 2}$ and $S(\alpha,n)$ fails then we have: $$(9 - \alpha \log 5)v_5 \rdima   \frac{\alpha\log 961 - 22}{2}v_{31}  + \frac{\alpha \log 81-13}{4}v_3 .$$
%(\alpha \log 7681 - 28)v_{7681} +  (\alpha \log 1921 - 24)\min(v_{17},v_{113}) 
\end{proposition}
\begin{proof} We again first burn $[4,5]^{v_5-1}n$, and then exhaust all powers of 2 which does not cost us anything for our desired range of $\alpha$ since $-1/5$ has just as many 1s as it has 0s in its 2-adic repeating part. We then use a transition gadget to allow us to reduce using $[1,81]$ repeatedly before using another transition gadget to move us onto the same track as above.

\end{proof}

\begin{corollary}If $\alpha \geq \frac{5}{2 \log 2}$ and $S(\alpha,n)$ fails then we have:

\label{7681-1921-961-481-241-81-41-long2-5cor}$$v_5 \rdima  0.433619038v_{31} + 0.22296285v_3  .$$
% FT Line 15
\end{corollary}
%1.33566697v_{7681} + 1.02313633406 \min(v_{17},v_{113})  +

\begin{proposition}

\label{3969-125-efficient2-31prop}If $\alpha \geq \frac{5}{2 \log 2}$ and $S(\alpha,n)$ fails then we have:

$$(21-\alpha \log 31) v_{31} \rdima \frac{\alpha \log 125 - 16}{3}v_5 + \frac{\alpha \log 3969 - 25}{2}\min(v_7,v_3/2)+ \frac{\alpha \log 32 - 11}{5}v_2  .$$

\end{proposition}
\begin{proof}
We burn $31$ with the straightforward initial gadget allowing us to then do repeated $[1,32]$ reduction, and then use a transition gadget to get to allow repeated $[1,125]$ reduction, which we can do since $124=4(31)$ followed by an additional transition gadget allowing for repeated $[1,3969]$  reduction. Here we are using that $3968=(3^7)(31)$. Note that $3969 = (3^4)(7^2)$.
\end{proof}
% If necessary Next term would be $v_{7937}$, and that works well for bumping since get $[1,7938]$ which is essentially $2(3^4)(7^2)$. 31*4096+1 is also pretty decent. 

\begin{corollary}

\label{125-63-efficient2-31cor} If $\alpha \geq \frac{5}{2 \log 2}$ and $S(\alpha,n)$ fails then we have:

$$ v_{31} \rdima 0.0547317v_5 +0.2836145 \min(v_7,v_3/2)+ 0.034824967v_2 .$$
% FT Line 17
\end{corollary}

\begin{comment}
Have $v_3 \rhd .55 v_7$, and  So have $\min(v_3/2,v_7) \rhd 0.275v_7$. Allows then $\min(v_3/2,v_7) \rhd 0.279v_7$ %  FT Line 19 

 and $\min(v_3/2,v_7) \rhd .047v_3$. %  FT Line 21 

Get then $v_3 \rhd .56v_7$, so $\min(v_3/2,v_7) \rhd 0.28v_7$ %FT Line 23 

Have then $v_5 \rhd  0.24v_3$ %FT Line 25 

Gives also $v_3 \rhd .32v_5$  %FT Line 27 

$v_3 \rhd .36v_5$ %FT Line 29 

$v_3 \rhd .373v_5$ %%FT Line 31

$v_7 \rhd .124v_3$ %FT Line 33

$v_5 \rhd  0.258v_3$ %FT Line 35

$v_3 \rhd  0.375v_5$ %FT Line 37 

Not using but note that now have $v_2 \rhd 1.27v_3$/ 
\end{comment}

\begin{proposition}

\label{194-193prop} If $\alpha \geq \frac{5}{2 \log 2}$ and $S(\alpha,n)$ fails then we have:

$$ (31- \alpha \log 193)v_{193} \rdima (\alpha \log 194 - 17)v_{97} . $$

\end{proposition}
\begin{proof} We use the following reduction: We repeatedly burn $193$ followed by an initial gadget and then do repeated $[1,194]$ reduction. Note that how much $[1,194]$ reduction we can do is controlled only by the size of $v_{97}$ since $v_2 \rhd v_{97}$ by Corollary \ref{4609-2305-577-289-193-97-49-25-9-2cor}. 

\end{proof}

%Would continue with 773-387 and then leads to 3089 which is prime and leads to [1,3090=2*3*5*103]). 

\begin{corollary} If $\alpha \geq \frac{5}{2 \log 2}$ and $S(\alpha,n)$ fails then we have:

\label{194-193cor}

$$v_{193} \rdima  0.1663870397v_{97} .$$
%FT Line 39 

% Next term would be 0.2072139\min(v_3/2,v_{43}) for this alpha value
\end{corollary}

Note that our current system is enough to conclude that $v_2 \rdima v_5$ for $\alpha  \geq \frac{5}{2 \log 2}$. However, this is not obvious simply from looking at these inequalities. We would like to be able to do so directly by treating our current system of inequalities as a simple linear programming problem. However, we run into the problem that we have variables that are of the form $\min(v_q,v_r)$ in addition to the various $v_p$ directly. Sometimes we will also want to understand terms of the form $\min(av_q,v_r)$ for some constant $a$.\\

Our solution is to bootstrap our system by adding in additional inequalities. For example,  in Corollary \ref{125-63-efficient2-31cor}, we have a term that is $\min(v_7,v_3/2)$. Using our current system of inequalities above, we can get that $v_3 \rdima .55 v_7$, and so $\min(v_3/2,v_7) \rdima 0.275v_7$. Including this as another inequality in our system then us to conclude that $\min(v_3/2,v_7) \rdima 0.279v_7$. We iterate this process as improving our inequalities involving minima of $v_3$,$v_5$ and $v_7$ to obtain that  $v_7 \rdima .124v_3$,  $v_5 \rdima  0.258v_3$, and  $v_3 \rdima  0.375v_3$ and apply them to the relevant minimum terms. This is then enough to conclude that $v_2 \rdima v_5$. We are now ready to prove:

\begin{proposition}
\label{233-2-30-29prop1} 
If $\alpha \geq \frac{5}{2 \log 2}$ and $S(\alpha,n)$ fails then we have:

$$(21 - \alpha \log 29)v_{29} \rdima (\alpha \log 233 - 18)v_{233} + (\alpha \log 30 - 11)\min(v_3,v_5) .$$

\end{proposition}
\begin{proof}
We burn $[28,29]$, and then reduce by $[1,30]$ as many times as we can which is dependent on $\min(v_3,v_5)$ since we already had $v_2 \rdima v_3$ and by our above remark we have $v_2 \rdima v_5$. We then reduce all our 2s, using that $-1/29$ in the 2-adics has many 2s as it has 1s, allowing us to then use a transition gadget which lets us repeatedly reduce by $[1,233]$.

\end{proof}

\begin{corollary} 

\label{233-2-30-29cor1} 
If $\alpha \geq \frac{5}{2 \log 2}$ and $S(\alpha,n)$ fails then we have:

$$v_{29} \rdima 0.187516256v_{233} + 0.143107\min(v_3,v_5)  . $$  %FT Line 41

\end{corollary}

% v_7 \rhd 0.135 v_3 %FT Line 43

% v_5 \rhd 0.263 v_3 %FT line 45

% v_3 \rhd 0.379 v_5 %FT line 47

% v_3 \rhd 0.381 v_5 %FT line 49

% v_7 \rhd 0.137 v_3 %FT Line 51

% v_3/2 \rhd 0.285 v_7 %FT Line 53

% v_3 \rhd 0.382 v_5 %FT line 55

Note that this system is now  enough to conclude that $v_3 \rdima .381v_5$, $v_5 \rdima .263v_3$, and $v_7 \rdima .137v_3$.\\

The next two pairs of propositions and corollaries are straightforward enough that we will omit their proofs:

\begin{proposition}

\label{467-234-233prop}If $\alpha \geq \frac{5}{2 \log 2}$ and $S(\alpha,n)$ fails then we have:

$$ (33-\alpha \log 233)v_{233} \rdima (\alpha \log 467 - 21)v_{467}+(\alpha \log 234 -17)\min(v_3/2,v_{13}) .$$

\end{proposition}

\begin{corollary}

\label{467-234-233cor} If $\alpha \geq \frac{5}{2 \log 2}$ and $S(\alpha,n)$ fails then we have:
$$v_{233} \rdima 0.08757403v_{467} + 0.2006\min(v_3/2,v_{13}) .$$ %  %FT line 57

\end{corollary}

% Next term is [1,1865], 1865=5*373, requires new prime, 373 but 374=2*11*17

\begin{proposition} If $\alpha \geq \frac{5}{2 \log 2}$ and $S(\alpha,n)$ fails then we have:
\label{935-468-467prop}
$$(41- \alpha \log 467)v_{467}\rdima (\alpha \log 935 - 23)\min(v_5,v_{11},v_{17}) + (\alpha \log 468 - 19)\min(v_3/2,v_{13}) \ldima v_{233} .$$ \end{proposition}

\begin{corollary} If $\alpha \geq \frac{5}{2 \log 2}$ and $S(\alpha,n)$ fails then we have:

\label{935-468-467cor}
$$v_{467} \rdima 0.08878896\min(v_5,v_{11},v_{17})+ 0.168646187\min(v_3/2,v_{13}) . $$
 %FT line 59
\end{corollary}

%Note that we have  $v_3 \rhd 0.28v_{13}$ and so $ \min(v_3/2,v_{13}) \rhd 0.14v_{13}$. %FT line 61 Note not immediately helpful since no lower bound v_{13} yet. 

Our next set of reductions yield additional lower bounds for $v_2$ and $v_3$:

\begin{proposition} If $\alpha \geq \frac{5}{2 \log 2}$ and $S(\alpha,n)$ fails then we have:

\label{769-193-97-49-25-9-2prop} 
$$(3-\alpha \log 2)v_2 \rdima (\alpha \log 769 - 21)v_{769} + (\alpha \log 193 - 17)v_{193} + (\alpha \log 97 - 15)v_{97}$$ $$ +\frac{\alpha \log 49 -13}{2}v_7 +\frac{\alpha \log 25 - 11}{2}v_5+ \frac{\alpha \log 9 - 7}{2}v_3 .$$

\end{proposition}
\begin{proof} This is identical to the start of the reduction of Proposition \ref{4609-2305-1153-577-289-193-97-49-25-9-2prop} but instead when we are done reducing by $[1,97]$ we use a transition gadget to do repeated $[1,193]$ reduction, and then use a transition gadget to do repeated $[1,769]$ reduction. 
\end{proof}

\begin{corollary}

\label{769-193-97-49-25-9-2cor} If $\alpha \geq \frac{5}{2 \log 2}$ and $S(\alpha,n)$ fails then we have:

$$ v_2 \rdima 5.9341989398v_{769}+  3.962285186 v_{193} + 2.99956421093v_{97} $$ $$ + 1.22821404v_7 +  0.75563395v_5+ 0.9248125036v_3 .$$ %FT line 63

\end{corollary}

Just as Proposition \ref{769-193-97-49-25-9-2prop} gave a modified version of Proposition \ref{4609-2305-1153-577-289-193-97-49-25-9-2prop}, we can give a modified version of Proposition \ref{4609-2305-1153-577-289-193-97-49-25-13-4-3prop}:

\begin{proposition}

\label{769-97-49-25-13-4-3prop} If $\alpha \geq \frac{5}{2 \log 2}$ and $S(\alpha,n)$ fails then we have:

$$(5-\alpha \log 3)v_3 \rdima (\alpha \log 769 - 21)v_{769}  + (\alpha \log 193 - 17)v_{193}  +  (\alpha \log 97 - 15)v_{97}$$

$$ +\frac{\alpha \log 49 -13}{2}v_7+\frac{\alpha \log 25 - 11}{2}v_5+ \frac{\alpha \log 4 - 5}{2}v_2.$$

\end{proposition}

\begin{corollary} If $\alpha \geq \frac{5}{2 \log 2}$ and $S(\alpha,n)$ fails then we have:

\label{769-193-97-49-25-13-4-3cor}

$$v_3 \rdima 2.8595965v_{769} +1.9093625 v_{193} + 1.4454425039v_{97} +0.49960527v_7+0.293776v_5 .$$ %FT line 65

\end{corollary}

Note that this allows us to conclude that $v_3 \rdima v_{769}$ which we will need when we prove Proposition \ref{2307-1154-1153prop} later.

We have  the straightforward lower bound on $v_{769}$: \begin{proposition}
\label{770-769prop} If $\alpha \geq \frac{5}{2 \log 2}$ and $S(\alpha,n)$ fails then we have:
$$(39-\alpha \log 769)v_{769} \rdima (\alpha \log 770 - 22)\min(v_5,v_7,v_{11}) .$$

\end{proposition} 
\begin{proof} This is essentially just the obvious reduction where one first repeatedly burns $v_{769}$ and then does repeated $[1,770]$. The only thing to note here is that to make this depend on just $\min(v_5,v_7,v_{11})$ we need to use that $v_2 \rdima v_5$.    

\end{proof}

%Next term is 1539=3^4*19 and then next term is 17*181
% If include can get direct dependence on 181  by using sets that go from 61 to do 361-81-61 alternate form, as well as from 361-181-91-31 

\begin{corollary} If $\alpha \geq \frac{5}{2 \log 2}$ and $S(\alpha,n)$ fails then we have:

\label{770-769cor}

$$ v_{769} \rdima 0.1311647\min(v_5,v_7,v_{11}). $$
%FT line 67
\end{corollary}

Our next proposition is also a straightforward reduction, along with the notes that $-1/577$ has as many 0s as 1s in the repeating its 2-adic expansion, and that $v_2 \rdima v_{17}$: 

\begin{proposition}  If $\alpha \geq \frac{5}{2 \log 2}$ and $S(\alpha,n)$ fails then we have:

\label{1155-long2-578-577prop}

$$(37 - \alpha \log 577)v_{577} \rdima (\alpha \log 1155 - 23)\min(v_3, v_5,v_7,v_{11}) + \frac{\alpha \log 578 -21}{2}v_{17}  . $$

\end{proposition}

%
%Note also may instead term 577*8+1=3^5*19 which is very nice. Note that next term in direct 577 is 2309 which is prime so leads to straightforward same min as usual. Add in above, and add in second type where use this second min. Note also that get a reasonably nice one from 512*577+1 = 3^2 * 5^2 * 13* 101. 

\begin{corollary}

\label{1155-long2--578-577cor} If $\alpha \geq \frac{5}{2 \log 2}$ and $S(\alpha,n)$ fails then we have:

$$ v_{577} \rdima 0.173019\min(v_3, v_5,v_7,v_{11}) +0.0688507 v_{17} .$$ %FT line 69

\end{corollary}

Our next proposition gives us a lower bound for $v_{17}$: \begin{proposition}

\label{137-35-18-17prop} If $\alpha \geq \frac{5}{2 \log 2}$ and $S(\alpha,n)$ fails then we have:

$$(17-\alpha \log 17)v_{17} \rdima (\alpha \log 137 - 17)v_{137} + (\alpha \log 35 - 12)\min(v_5,v_7)  + \frac{\alpha \log 18 - 9}{2}v_3 .$$

\end{proposition}
\begin{proof} Again, the reduction needed is straightforward, requiring only the note that $-1/17$ has many 0s as 1s in its repeating part of its 2-adic expansion. 

\end{proof}

\begin{corollary}

\label{137-35-18-17cor} If $\alpha \geq \frac{5}{2 \log 2}$ and $S(\alpha,n)$ fails then we have:

$$v_{17} \rdima 0.109872v_{137} + 0.12139299\min(v_5,v_7)+ 0.1050538v_3 .$$ %FT line 71

\end{corollary}

\begin{comment}
$\min(v_5,v_7) \rhd 0.301v_7$ %FT line 73
$\min(v_5,v_7) \rhd 0.327v_7$ %FT line 75
$\min(v_5,v_7) \rhd 0.188v_5$ %FT line 77
$\min(v_5,v_7) \rhd 0.192v_5$ %FT line 79
$\min(v_3,v_5) \rhd 0.269v_3$ %FT line 81
$\min(v_3,v_5) \rhd 0.491v_5$ %FT line 83
$\min(v_3/2,v_{13}) \rhd 0.186v_{13}$ %FT line 85
\end{comment}

We have then:
\begin{proposition}

\label{2307-1154-1153prop} If $\alpha \geq \frac{5}{2 \log 2}$ and $S(\alpha,n)$ fails then we have:

$$(41 - \alpha \log 1153)v_{1153} \rdima (\alpha \log 2307- 25)\min(v_{769}) +  (\alpha \log 1154 -22)v_{577} .$$

\end{proposition}
\begin{proof} This is a straightforward reduction with two minor notes. The reduction is to first repeatedly burn $1153$, then do repeated $[1,1154]$ repeatedly, and then do $[2,2307]$ repeatedly noting that $2307=(767)(3)$ and that $v_3 \rdima v_{769}$ as noted earlier. We need to do $[2,2307]$ rather than our usual method of using powers of $2$ and then slipping in an extra division by 2 because the 2-adic repeating part of $-1/1153$ is too long for us to verify that the number of $0$s is at least the number of $1$s.  

\end{proof}

\begin{corollary}

\label{2307-1154-1153cor} If $\alpha \geq \frac{5}{2 \log 2}$ and $S(\alpha,n)$ fails then we have:

$$v_{1153} \rdima 0.18812579v_{769} +  0.22033497v_{577} .$$
%FT line 87
\end{corollary}

\begin{comment}
Note at this point have $v_2 \rhd 1.77v_3$. 
$\min(v_3,v_5) \rhd 0.497v_5$ %FT line 89
$\min(v_5,v_7) \rhd 0.193v_5$ % FT line 91
$\min(v_3,v_5) \rhd 0.272v_3$ %FT line 93
$\min(v_5,v_7) \rhd 0.233v_3$ % FT line 95
$\min(v_5,v_7) \rhd 0.237v_3$ % FT line 97
$\min(v_3,v_5) \rhd 0.273v_3$ %FT line 99
$\min(v_5,v_7) \rhd 0.333v_7$ %FT line 101

\end{comment}

We have the following straightforward reduction:

\begin{proposition}
\label{11propsimple2}  If $\alpha \geq \frac{5}{2 \log 2}$ and $S(\alpha,n)$ fails then we have:

$$(15- \alpha \log 11)v_{11} \rdima (\alpha \log 12 - 8)\min(v_2/2,v_3) .$$
\end{proposition}

\begin{corollary}
\label{11corsimple2}  If $\alpha \geq \frac{5}{2 \log 2}$ and $S(\alpha,n)$ fails then we have:

$$ v_{11} \rdima 0.15152613\min(v_2/2,v_3). $$ %FT line 103
\end{corollary}

\begin{comment}
$\min(v_2/2,v_3) \rhd 0.91v_3$. %FT line 105
$\min(v_2/2,v_3) \rhd 0.93v_3$. %FT line 107
$\min(v_3,v_5,v_7,v_{11}) \rhd 0.237v_3$ % FT line 109
$\min(v_3,v_5,v_7,v_{11}) \rhd 0.249v_3$ % FT line 111
$\min(v_3,v_5,v_7,v_{11}) \rhd 0.251v_3$ % FT line 113
$\min(v_5,v_7,v_{11}) \rhd 0.251v_3$ % FT line 115
$\min(v_3,v_5,v_7,v_{11}) \rhd 0.253v_3$ % FT line 117
$\min(v_5,v_7,v_{11}) \rhd 0.253v_3$ % FT line 119
$\min(v_2/2,v_3) \rhd v_3$. %FT line 121 FINAL
$\min(v_5,v_7) \rhd 0.209v_5$ % FT line 123
$\min(v_3,v_5) \rhd 0.3692v_3$ %FT line 125
$\min(v_3,v_5) \rhd 0.644v_5$ %FT line 127
$\min(v_3,v_5) \rhd 0.713v_5$ %FT line 129
$\min(v_3,v_5) \rhd 0.732v_5$ %FT line 131
$\min(v_3,v_5) \rhd 0.737v_5$ %FT line 133
$\min(v_3,v_5) \rhd 0.738v_5$ %FT line 135
$\min(v_3,v_5) \rhd 0.739v_5$ %FT line 137
$\min(v_5,v_7,v_{11}) \rhd \min(v_3,v_5,v_7,v_{11})$ %Convenience FT line 139
$\min(v_5,v_7) \rhd 0.339v_5$ % FT line 141
$\min(v_5,v_7) \rhd 0.343v_5$ % FT line 143
$\min(v_3,v_5) \rhd 0.777v_5$ %FT line 145
$\min(v_3,v_5) \rhd 0.787v_5$ %FT line 147
$\min(v_3,v_5) \rhd 0.789v_5$ %FT line 149
$\min(v_5,v_7) \rhd 0.361v_5$ % FT line 151
$\min(v_5,v_7) \rhd 0.362v_5$ % FT line 153
$\min(v_3,v_5) \rhd 0.796v_5$ %FT line 155
$\min(v_3,v_5) \rhd 0.797v_5$ %FT line 157
$\min(v_3,v_5) \rhd 0.798v_5$ %FT line 159
$\min(v_5,v_7) \rhd 0.365v_5$ % FT line 161

\end{comment}

The next proposition is again a simple reduction:

\begin{proposition}

\label{27-14-13prop}  If $\alpha \geq \frac{5}{2 \log 2}$ and $S(\alpha,n)$ fails then we have:

$$(15-\alpha \log 13)v_{13} \rdima \frac{\alpha \log 27 - 10}{3}v_3  +(\alpha \log 14 -9)v_7 .$$

\end{proposition}
\begin{proof}
The reduction here is straightforward, burning $[12,13]$ repeatedly, then using $[1,14]$ reduction before switching to $[1,27]$ reduction. The only thing to note is that how many times we can do $[1,14]$ reduction depends only on $v_7$ because $v_2 \rdima v_7$. 
\end{proof}

\begin{corollary}

\label{27-14-13cor}  If $\alpha \geq \frac{5}{2 \log 2}$ and $S(\alpha,n)$ fails then we have:

$$v_{13} \rdima 0.1094249v_{3} + .090171v_7 .$$
% FT line 163
\end{corollary}

\begin{comment}
$\min(v_3,v_5) \rhd 0.811v_5$ %FT line 165
$\min(v_3,v_5) \rhd 0.815v_5$ %FT line 167
$\min(v_3,v_5) \rhd 0.816v_5$ %FT line 169
$\min(v_3/2,v_{13}) \rhd 0.353 v_{13}$ % FT line 171
$\min(v_5,v_7) \rhd 0.373v_5$ % FT line 173

\end{comment}

Using our bootstrapping to handle the minima again, we can now conclude that $v_2 \rdima 1.01(v_5+v_7)$ (In fact can easily get $v_2 \rdima 1.9(v_5+v_7)$ which allows us to have enough 2s to do the next proposition:

\begin{proposition}

\label{170-27-14-13prop}  If $\alpha \geq \frac{5}{2 \log 2}$ and $S(\alpha,n)$ fails then we have:

$$(15-\alpha \log 13)v_{13} \rdima (\alpha \log 170 - 17)\min(v_5,v_{17}) + \frac{\alpha \log 27 - 11}{3}v_3  +(\alpha \log 14 -9)v_7 .$$

\end{proposition}
\begin{proof} This reduction is similar to the previous reduction except that instead of exhausting the 2s after doing all the $[1,14]$ reduction we do repeated $[2,27]$ until we slip in an extra division by $13$ followed by $[1,170]$ reduction.  

\end{proof}

\begin{corollary}  If $\alpha \geq \frac{5}{2 \log 2}$ and $S(\alpha,n)$ fails then we have:
$$v_{13} \rdima 0.265\min(v_5,v_{17})+ 0.0514428v_{3} + 0.090171v_7 .$$ %FT line 175

\end{corollary}

\begin{proposition}
\label{290-18-17prop} If $\alpha \geq \frac{5}{2 \log 2}$ and $S(\alpha,n)$ fails then we have:

$$(17-\alpha \log 17)v_{17} \rdima (\alpha \log 290 -19)\min(v_5,v_{29})   + \frac{\alpha \log 18 - 9}{2}v_3.$$
\end{proposition}
\begin{proof} This is a straightforward reduction where we burn $[16,17]$ repeatedly to then do $[1,18]$ repeatedly, slip in an extra division by $17$, and then do $[1,290]$ repeatedly. The only quirk we need to note is that we have to use further bootstrapping to verify that we have sufficient $v_2$, in particular that  $v_2 \rhd 0.5v_3 + v_5$.
\end{proof}

\begin{corollary} If $\alpha \geq \frac{5}{2 \log 2}$ and $S(\alpha,n)$ fails then we have:
\label{290-18-17cor}
$$v_{17} \rdima 0.21378844\min(v_5,v_{29}) +0.1050538v_3 .$$ %FT line 177
\end{corollary}

We also have the following lower bound for $v_{97}$:

\begin{proposition} If $\alpha \geq \frac{5}{2 \log 2}$ and $S(\alpha,n)$ fails then we have:

\label{389-195-long2-98-97prop}

$$(27- \alpha \log 97)v_{97} \rdima (\alpha \log 389 - 20)v_{389} +  (\alpha \log 195 -17)\min(v_3,v_5,v_{13}) +\frac{\alpha \log 98 - 15}{2}v_7.$$

\end{proposition}
\begin{proof}
We burn $97$ repeatedly followed by repeating $[1,98]$, then slip in an additional division by 2,  and then repeat $[1,195]$ followed by repeating $[1,389]$.

\end{proof}

\begin{corollary} If $\alpha \geq \frac{5}{2 \log 2}$ and $S(\alpha,n)$ fails then we have:

\label{389-195-long2-98-97cor}

$$v_{97} \rdima 0.1437175759v_{389} + 0.19221751\min(v_3,v_5,v_{13}) + 0.0731782v_7 .$$ %FT line 179

\end{corollary}

\begin{comment}
$\min(v_5,v_{29}) \hd 0.115v_5 $ %FT line 181
$\min(v_5,v_{29}) \hd 0.038v_{29} $ %FT line 183
$\min(v_3,v_5,v_{13} \rhd 0.253v_{13} $ %FT line 185
$\min(v_3,v_5,v_{13} \rhd 0.371v_{13} $ %FT line 187
$\min(v_3,v_5,v_{13} \rhd 0.427v_{13} $ %FT line 189
$\min(v_5,v_7) \rhd 0.386v_5$ % FT line 191
$\min(v_3,v_5) \rhd 0.933v_5 $ % FT line 193
$\min(v_3,v_5) \rhd 0.976v_5 $ % FT line 195
$\min(v_3,v_5) \rhd 0.992v_5 $ % FT line 197
$\min(v_3,v_5) \rhd 0.998v_5 $ % FT line 199
$\min(v_3,v_5) \rhd 0.382v_3 $ % FT line 201
$\min(v_5,v_{29}) \hd 0.141v_5 $ %FT line 203
$\min(v_5,v_{29}) \hd 0.051v_{29} $ %FT line 205
$\min(v_5,v_{29}) \hd 0.057v_{29} $ %FT line 207
$\min(v_3,v_5) \rhd 0.999v_5 $ % FT line 209
$\min(v_5,v_{29}) \hd 0.142v_5 $ %FT line 211
$\min(v_3,v_5,v_{13} \rhd 0.453v_{13} $ %FT line 213
$\min(v_3,v_5,v_{13} \rhd 0.465v_{13} $ %FT line 215
$\min(v_3,v_5,v_{13} \rhd 0.471v_{13} $ %FT line 217
$\min(v_3,v_5,v_{13} \rhd 0.474v_{13} $ %FT line 219
$\min(v_3,v_5,v_{13} \rhd 0.475v_{13} $ %FT line 221
$\min(v_5,v_7) \rhd 0.455v_5$ % FT line 223
$\min(v_5,v_7) \rhd 0.457v_5$ % FT line 225
Now have $v_3 \rhd v_5$ So get $\min(v_3,v_5) \rhd v_3 $ % FT line 227 final. 
\end{comment}

\begin{proposition}
\label{390-389prop} If $\alpha \geq \frac{5}{2 \log 2}$ and $S(\alpha,n)$ fails then we have:
$$(37-\alpha \log 389)v_{389} \rdima  (\alpha \log 390 -19)\min(v_3,v_5,v_{13}) . $$
\end{proposition}
\begin{proof} 
Proof omitted. 
\end{proof}

% Next term would be (\alpha \log 779 - 22)\min(v_{19},v_{41}) % with 0.129992*(\alpha \log 779 - 24)\min(v_{19},v_{41})

\begin{corollary} If $\alpha \geq \frac{5}{2 \log 2}$ and $S(\alpha,n)$ fails then we have:
\label{390-389cor}
$$v_{389} \rdima 0.16256771\min(v_3,v_5,v_{13}) . $$
\end{corollary} %FT line 229

\begin{comment}
$v_5 \rhd v_{17}$ % FT line 231 
$\min(v_5,v_7) \rhd 0.461v_5$ % FT line 233
$\min(v_3,v_5,v_7,v_{11}) \rhd 0.461v_5 $ %FT line 235 
$\min(v_5,v_7) \rhd \min(v_3,v_5,v_7,v_{11})$ % Convenience FT line 237 
$\min(v_3,v_5,v_7,v_{11}) \rhd 0.482v_5 $ %FT line 239 

$\min(v_3,v_5,v_7,v_{11}) \rhd 0.485v_5 $ %FT line 241 
$\min(v_3,v_5,v_7,v_{11}) \rhd 0.486v_5 $ %FT line 243
$\min(v_5,v_{11},v_{17} \rhd 0.077v_{11} $ % FT 245
$\min(v_5,v_{11},v_{17} \rhd v_{17} $ % FT 247 Final

\end{comment}

\begin{proposition}
\label{839-420-419prop}  If $\alpha \geq \frac{5}{2 \log 2}$ and $S(\alpha,n)$ fails then we have:
$$(37 - \alpha \log 419)v_{419} \rdima (\alpha \log 839 - 23)v_{839} +  (\alpha \log 420 - 19)\min(v_5,v_7) .$$
\end{proposition}
\begin{proof}
The reduction here is to first burn $419$ repeatedly and then do repeated $[1,420]$ reduction before slipping in a 2 to allow repeated $[1,839]$ reduction. We note that $v_2 \rdima 2v_7$ and so $[1,420]$ reduction only depends on $v_5$ and $v_7$. 
\end{proof}

\begin{corollary}
\label{839-420-419cor}  If $\alpha \geq \frac{5}{2 \log 2}$ and $S(\alpha,n)$ fails then we have:
$$v_{419} \rdima 0.084169v_{839} + 0.182987\min(v_5,v_7) .$$ % FT line 249 
\end{corollary}

We can bootstrap to get that $v_2 \rdima 3v_7$ and so we are able to prove:

\begin{proposition}
\label{840-839prop} If $\alpha \geq \frac{5}{2 \log 2}$ and $S(\alpha,n)$ fails then we have:
$$ (45 - \alpha \log 839)v_{839} \rdima  (\alpha \log 840 -21)\min(v_5,v_7)  .$$ 
\end{proposition}
% 1679 would be the next term. Is 23*79

\begin{corollary} If $\alpha \geq \frac{5}{2 \log 2}$ and $S(\alpha,n)$ fails then we have:
$$v_{839} \rdima  0.158582\min(v_5,v_7) . $$
\end{corollary} %line 251

\begin{comment}

$\min(v_{11},v_{419}) \rhd 0.037v_{11}$. % FT line 253

$\min(v_{11},v_{419}) \rhd 0.051v_{11}$. % FT line 255

$\min(v_{11},v_{419}) \rhd 0.057v_{11}$. % FT line 257

$\min(v_{11},v_{419}) \rhd 0.059v_{11}$. % FT line 259

$\min(v_{11},v_{419}) \rhd 0.061v_{11}$. % FT line 261

$\min(v_{11},v_{419}) \rhd 0.062v_{11}$. % FT line 263

$\min(v_{11},v_{419}) \rhd 0.062v_{11}$. % FT line 263

$\min(v_{11},v_{419}) \rhd 0.048v_3$. % FT line 265

\end{comment}

Note that we easily have from our above inequalities that $v_5 \rdima v_{97}$. 

\begin{proposition}

\label{485-12-11prop} If $\alpha \geq \frac{5}{2 \log 2}$ and $S(\alpha,n)$ fails then we have:

$$(15- \alpha \log 11)v_{11} \rdima (\alpha \log 485 - 20)v_{97} + (\alpha \log 12 - 8)\min(v_2/2,v_3) .$$

\end{proposition}
\begin{proof} We burn $[10,11]$ repeatedly, and then do $[1,12]$ reduction repeatedly before slipping in an extra division by $2$. We then slip in an extra division by 11 and another $2$ allowing us to do $[1,485]$ reduction. We note that how many times we can reduce by $[1,485]$ is a function just of $v_{97}$ since by our earlier remark we have $ v_5 \rdima v_{97}$.  

\end{proof}

\begin{corollary} If $\alpha \geq \frac{5}{2 \log 2}$ and $S(\alpha,n)$ fails then we have:

\label{485-12-11cor}

$$ v_{11} \rdima 0.362848v_{97}+ 0.15152613v_3 . $$ % FT line 267

\end{corollary}
%Note that the other 11 type where $[1,243]$ followed by $[2,485]$ is weaker on both terms since gets a v_3 coefficient of around .12 so not helpful. 

%Next term of this would be [1,3*17*19] so better if subtract off from first term. Term after would be $[1,13*149]$. A fair bit later get 3 * 17 * 43 * 113 which is also strong.

% A different option is to try increasing both 11 and 2 after 484 term. This branch gives [1,23*463]. 463 is a new prime but not as bad as might think since 464=2^4 * 29.
%If continue that path 484*22^4+1 =  5 * 97 * 157 * 1489 so get strictly better term if use min to get as many of that form as possible and use remaining [1,5*97] for left-overs. 
%If were to do this, then should search for all variants on 11 and 2 powers rather than just 22 together. For example, 484*11^3*2+1 =   19^2 * 43 * 83 

A careful reader may wonder why we prefer $[1,12]$ reduction in \ref{485-12-11prop}, rather than using our 3s by doing $[1,243]$. In fact, $[1,243]$ yields a smaller coefficient for $v_3$ in the inequality for the range of $\alpha$ we care about, and one would have only $0.12v_3$. As $\alpha$ gets smaller,   $[1,243]$ gets more efficient and there is a point where it will surpass $[1,12]$.

We are now in a position to prove Theorem \ref{alpha2.5ln2}. \begin{proof} Assume that $S(\alpha,n)$ fails for $\alpha \geq \frac{5}{2 \log 2}$ and that $v_p$ is sufficiently large for some set of primes $p$. Then all the linear inequalities in this section are true. However, with repeated bootstrapping this system is now sufficient to get a contradiction. 

\end{proof} 

Note that in theory the method given here is constructive since one can use Dirichlet's theorem to guarantee the existence of the necessary transition gadgets. The inequalities involved use the primes 2, 3, 5, 7, 11, 13, 17, 29, 31, 97, 193, 233, 389, 419, 467, 577, 769, 839 and 1153. Their product is slightly under $10^{36}$. Recall Linnik's theorem: Let $P(a,d)$ be the  the least prime in the arithmetic progression $a, a+d, a+2d \cdots $. Then Linnik's theorem says there are effective constants $c$ and $\ell$ with $p(a,d) \leq cd^{\ell}$ as long as $a < d$. Thus, if one has effective constants in Linnik's theorem, one can then take $S$ to be the set of primes that are at most $c(10^{36})^{\ell}$ but where one only needs $v_p \geq 1$ for all the primes except the 19 primes used in our linear inequalities who will require larger values. It is plausible that a careful and extensive computation could use a smaller set. We suspect that our set of 19 primes above can be taken as a valid set for $S$ with no other primes.

\section{P-adic expansions and efficient reductions}

In the previous section, we had to take substantial advantage of primes $p$ where $-1/p$ had many 0s in the repeating parts of their 2-adic expansion. In the section prior, we also took advantage of such primes albeit in a more limited fashion In particular in the previous section, in order to start a reduction which required slipping in an extra division by 2, we had to start with a prime $p$ such that the number of zeros  in the repeating parts is at least as high as the number of 1s. When these two quantities are exactly equal, one does not end up with any $v_2$ term on the right-hand side of our resulting inequality when $\alpha = \frac{2.5}{\log 2}$, and we get inequalities with a $v_2$ term with positive coefficient when there are more $0$s than $1$s. If one were to attempt to make the same reduction with a prime $p$ that has more 1s than 0s, one would have a coefficient with a negative reduction which would substantially complicate matters. Luckily, such primes are very rare. 

Motivated by these observations we will say an odd number $n>1$ is \emph{2-balanced} (or just \emph{balanced}) if $\frac{-1}{n}$ has in its purely periodic 2-adic piece exactly as many 1s as 0s. Similarly, we say  $n$ is \emph{2-efficient} (or just \emph{efficient}) if it has more 0s than 1s, and 2-inefficient (or just inefficient) if it has more 1s than 0s. The first thing to note is that 2-inefficient numbers do exist but they seem to be very rare. The smallest example is $187$  which has 21 1s and only 19 0s. 2-inefficient primes also do exist but are very rare; an example is $p=937$ where there are 59 1s and 58 0s.\\

There are good heuristic reasons for expecting that inefficient numbers and inefficient primes should be uncommon. Let $n$ be a positive odd number greater than $1$. Note that we can obtain the 2-adic expansion of $-1/n$ explicitly by taking  $$\frac{2^{k\phi(n)}-1}{n}$$ in base $2$ and letting $k$ go to infinity. (Here $\phi$ is the Euler $\phi$ function.)  Moreover, in the case of $k=1$ we will just have the quantity  $$M=\frac{2^{\phi(n)}-1}{n}.$$ Note that $M$'s base $2$ expansion must contain a complete copy of the repeating part of $-1/n$, and in fact must agree with the first $\phi(n)$ digits of $-1/n$, but the last set of digits of $M$ must be all 0s, because $M$ has about $\log_2 n$ fewer digits than $2^{\phi(n)}$ does, so there's a bias towards having extra zeros.\\

We have the following questions about efficient, balanced and inefficient numbers.\\ 

 Let $B(x$) be the number of balanced n less than or equal to x, and similarly define $E(x)$ and $I(x)$. 
 
 \begin{conjecture} For all $x$ we have $B(x) +E(x) \geq I(x)$.
 \end{conjecture}
 
 It is also likely that for all $x$ we have that $B(x) \geq I(x)$ and that $E(x) \geq I(x)$, but we are substantially less confident in those claims and so will not label them as conjectures. Asymptotics for all three of these quantities  would be nice. Unfortunately, at this point we do not even have conjectural orders of magnitude of growth for any of these three quantities, although we are confident that the set of all odd numbers which are either balanced or efficient  has density 1 in the set of odd numbers, and thus strongly suspect that $I(x)=o(x)$.
 
 Unfortunately, at present, we do not even see any way to prove that there are infinitely many inefficient numbers. We have easily that there are infinitely many  efficient numbers; if we set $n=2^k-1$ then $-1/n$ has repeating part with exactly one 1 and $k-1$ zeros. Thus, $2^k-1$ is efficient for $k \geq 2$.  Similarly, $n=2^k+1$ is easily seen to be balanced.  

Of course, for our earlier situation, what we really care about are not just efficient numbers, balanced numbers, and inefficient numbers, but rather efficient primes, balanced primes, and inefficient primes.  One obvious question is are there are infinitely many primes which are balanced, efficient or inefficient? We strongly expect that there are infinitely many efficient primes.  Note that if there are infinitely many Mersenne primes then this is trivially true by the remark in the last paragraph. Since the set of efficient primes appears to be much larger than the set of Mersenne primes, one would hope that an unconditional proof is possible. A subsequent paper by this author and Christian Roettger will show that there are infinitely many balanced primes. \cite{Roettger}

Another natural question is whether powers of primes are efficient, balanced or inefficient. In the context of the original problem, the most natural question is whether all powers of $3$ are balanced. A power of 3 that is efficient would be very helpful for improving Theorem \ref{alpha2.5ln2}/ An efficient power of $3$ would also allow one to possibly give a shorter proof of Theorem \ref{alpha2.5ln2} by allowing $v_3$ to be directly bounded below by a function of $v_2$ rather than having to go through other $v_p$ as intermediaries. Such a power of $3$ if it is not too big and is reasonably efficient might also allow for a shorter proof of Theorem \ref{upperbounddirectimprovement} or result in tightening of that Theorem. Unfortunately, no such power of 3 exists. The aforementioned paper by \cite{Roettger} will prove that all powers of 3 are balanced. 

Our questions about balanced, efficient, and inefficient numbers can also be generalized to the $p$-adics for other primes $p$. However, there are multiple potential generalizations. We discuss three of the more natural ones  here.

 First, one can compare the number of non-zero and zero digits in which case one has an expected ratio based on a naive random model where the chance that any given digit is zero is $1/(p-1)$. One then correspondingly defines numbers a  as balanced, efficient and inefficient with respect to the expected number of 0s. However, this seems tricky and unlikely to be the correct generalization. At a practical level, this generalization seems to be unhelpful for understanding integer complexity.  
 
  Second,  one can take the average value of the digits in the purely periodic part, and defines a number to be $p$-balanced if the average is exactly $(p-1)/2$, $p$-efficient if the average is strictly less than $(p-1)/2$ and $p$-inefficient if the average is great $(p-1)/2$. This generalization seems to be the most natural generalization.
  
  The third possible generalization is a variant of the second generalization where instead of averaging the values of the digits, one instead averages the integer complexity of the values of the digits. This average is a useful metric for understanding whether a given reduction is likely to give rise to a productive inequality using our methods above. Unfortunately, it isn't obvious how to use this average to define corresponding notions of balanced, efficient, and inefficient.
  
  It is most likely that that second generalization is the  most natural to anyone interested in questions about $p$-adic behavior independent of their applications to integer complexity issues. The second generalization also has the advantage that as long as $p$ is small in practice this will also give useful information from an integer complexity standpoint. 
 
Define $S(m,n,\alpha)$ to be  to denote the statement ``If for all $m<k<n$, $\cpx{k} \leq \alpha \log k$ then $\cpx{n} \leq  \alpha \log n$.'' We suspect then that a careful study of efficient primes will allow one to prove the following:

\begin{conjecture} For any $\alpha > \frac{2}{\log 2}$ there is a finite set of primes $S_\alpha$ such that for any $m$ if $v_p$ is sufficiently large for all $p \in S_\alpha$, then $S(m,n,\alpha)$.
\label{biginductionconjecture}
\end{conjecture}

Note that with a small amount of additional work the proof of Theorem \ref{alpha2.5ln2} can be adjusted to give a stronger result in terms of $S(m,n,\alpha)$ and can then be thought of as the special case of Conjecture \ref{biginductionconjecture} where $\alpha = \frac{5}{2\log 2}$. 
 
One question closely related to the above conjecture is whether these methods can be used to establish results of the form For all $n$ except an explicit, finite exceptional set B, we have $\cpx{n} \leq \beta \log n$ where $\beta$ is some number strictly less than $\frac{26}{\log 1439}$. While our results have been phrased in terms of the interval $I_0$, they can mostly be adjusted with little work to become sensible reductions for smaller choices of $\alpha$, and it may be possible to make such an extension. Note that in this case the definite results may be easier to handle than the indefinite results. Suppose for example one is interested in the case $n \equiv 1$ (mod $6$). Then the reduction $[1,6]n$ involves a the number $\frac{n-1}{6}$ and so one can have that reduction be used to get an efficient representation of $n$ as long as $n$ is larger than some exceptional set. For example, if one had as the exceptional set just $\{1, 1439\}$ it  would be possibly useful as long as one had $n> 6(1439)+1$. The difficulty then would arise from the indefinite results, and one would need to have additional auxiliary primes to ensure that the reduction did not lead to too small a number.

\section{Appendix}

Proof of Lemma{\ref{v7=0}}:
\begin{proof}  This proof is very similar to the earlier lemmata.

We may assume that $n \equiv 1,2,3,4$ or $5$ (mod 7). And by the above work, we may also assume
assume that $n \equiv 8 $ (mod 9), $n \equiv 4$ (mod 5) and $n \equiv 15$ (mod 16). In fact, we can do better than that for $v_2$, but will not need it in this lemma. As before we have cases, this time based on $n$ (mod 7):\\

Case {\bf I}: If $n \equiv 1$ (mod 7) then we may use $[1,14]n$.\\

Case {\bf II}: If $n \equiv 2$ (mod 7) then we may use $[2,21]n$.\\

Case {\bf III}: If $n \equiv 3$ (mod 7). then we may use $[1,4][1,10][1,3][1,7][1,2]n$.\\

Case {\bf IV}: If $n \equiv 4$ (mod 7) then we may use $[1,4][1,10][1,3][1,7][1,2][2,3]n$.\\ %Note that this is the worst case for this lemma.

Case {\bf V}: If $n \equiv 5$ (mod 7) then we may use $[1,8][1,14][2,3]n$.
\end{proof}

Proof of Lemma \ref{v13=0}:

\begin{proof} We may assume that  $v_2 \geq 6$, $v_3 \geq 2$, $v_5 \geq 1$, $v_7 \geq 1$.

Case {\bf I}. $n \equiv 1 $ (mod 13). We may use $[1,26]n$.\\

Case {\bf II}, $n \equiv 2 $ (mod 13). We may use $[2,39]n$.\\

Case {\bf III}. $n \equiv 3 $ (mod 13). We may use $[1,26][1,2]n$.\\

Case {\bf IV}. $n \equiv 4 $ (mod 13). We may use $[1,8][2,9][1,14][4,65]n$.\\

Case {\bf V}. $n \equiv 5 $ (mod 13). We may use $[1,14][2,39][1,2]n$.\\

Case {\bf VI}. $n \equiv 6 $ (mod 13). We may use one of $[1,48][2,9][6,91]n$,\\

$[1,6][1,16][2,9][6,91]n$ or $[1,6][1,10][1,2][1,16][2,9][6,91]n$ or $[1,20][1,2][1,16][2,9][6,91]n$, depending on $[1,16][2,9][6,91]n$ (mod 6).\\

Case {\bf VII}. $n \equiv 7$ (mod 13). We may use $[1,26][1,2]^2n$.\\

Case {\bf VIII}. $n \equiv 8$ (mod 13). We may use $[1,14][2,39][2,3]n$.\\

Case {\bf IX}. $n \equiv 9$ (mod 13). We may use $[1,8][2,9][1,14][4,65][1,2]n$.\\

Case {\bf X}. $n \equiv 10$ (mod 13). We may use $[1,14][2,39][1,2]^3n$.\\

Case {\bf XI}. $n \equiv 11$ (mod 13). May use $[1,14][1,26][1,2][2,3]n$.\\
\end{proof}

Proof of Lemma \ref{v11=0}:
\begin{proof}

We may assume that $v_2 \geq 8$, $v_3 \geq 3$, $v_5 \geq 1$, $v_7 \geq 1$ and $v_{13} \geq 1$.

Case {\bf I} $n \equiv 1$ (mod 11). We may use $[1,12]^3[1,22]n$.\\

Case {\bf II}: $n \equiv 2$ (mod 11). We may use $[1,4][1,12]^2[2,33]n$.\\

Case {\bf III}: $n \equiv 3$ (mod 11). We may use $[1,12]^3[1,22][1,2]n$.\\

Case {\bf IV}: $n \equiv 4$ (mod 11). We may use $[1,14][1,4][1,12]^2[4,55]n$.\\

Case {\bf V}: $n \equiv 5$ (mod 11).  We may use $[1,4][1,12]^2[2,33][1,2]n$.\\

Case {\bf VI}: $n \equiv 6$ (mod 11). We may use $[1,12]^3[6,77]n$.\\

Case {\bf VII}: $n \equiv 7$ (mod 11). We may use $[1,12]^2[1,22][1,2]^2n$.\\

Case {\bf VIII}: $n \equiv 8$ (mod 11). We may use  $[1,14][1,4][1,12][2,33][2,3]n$.\\

Case {\bf IX}: $n \equiv 9$ (mod 11). We may use  $[1,14][1,4][1,12]^2[4,55][1,2]n$. \\

Note that in Cases IV and IX we did not need to use that $\cpx{55}=12$ from $55=2(27)+1)$ rather than the weaker representation from the factorization of $55$, so if the earlier lemmata can be tightened, this one is easier to tighten than it might naively appear.

\end{proof}

Proof of Lemma \ref{v19=0}

\begin{proof}

We may assume that $v_2 \geq 8$, $v_3 \geq 3$, $v_5 \geq 1$, $v_7 \geq 1$, $v_{11} \geq 1$, and $v_{13} \geq 1$.

Case {\bf I}: $n \equiv 1 $ (mod 19). We may use $[1,20][1,38](n)$.\\

Case {\bf II} $n \equiv 2 $ (mod 19). We may use $[1,20][2,57](n)$.\\

Case {\bf III} $n \equiv 3 $ (mod 19). We may use $[1,10][1,38][1,2](n)$.\\

Case {\bf IV}: $n \equiv 4 $ (mod 19). We may use $[1,3][1,12][1,4][4,95](n)$.\\

Case {\bf V}: $n \equiv 5$ (mod 19). We may use $[2,57][1,2](n)$.\\

Case {\bf VI}: $n \equiv 6$ (mod 19). We may use $[1,24][1,4][6,133](n)$.\\

Case {\bf VII}: $n \equiv 7$ (mod 19). We may use $[1,20][1,38][1,2]^3n$.\\

Case {\bf VIII}: $n \equiv 8$ (mod 19). We may use $[1,20][2,57][2,3]n$.\\

Case {\bf IX}: $n \equiv 9$ (mod 19). We may use $[1,21][1,12][1,4][4,95][1,2]n$.\\

Case {\bf X}: $n \equiv 10$ (mod 19). We may use  $[1,22][1,8][1,4][2,39][2,57][1,2][4,5]n$.\\

Case {\bf XI}:$n \equiv 11$ (mod 19). We may use $[2,57][1,2]^2n$.\\

Case {\bf XII}: $n \equiv 12$ (mod 19). We may use $[A][1,4][1,27][1,4][1,4][1,20][12,247]n$.\\
Here $[A]$ is either  $[1,13][0,2]$ or  $[1,6]$.\\

Case {\bf XIII}: $n \equiv 13$ (mod 19). We may use $[1,24][1,4][6,133][1,2]n$.\\
%3.744

Case {\bf XIV}: $n \equiv 14$ (mod 19).  We may use $[1,3][1,24][1,4][4,95][2,3]n$.\\
% 3.711 

Case {\bf XV}: $n \equiv 15$ (mod 19). We may use $[1,10][1,38][1,2]^4n$.\\
%3.672

Case {\bf XVI}: $n \equiv 16$ (mod 19) We may use either $[1,8][1,4][2,57][2,3][1,2][1,2]n$ \\ 
 \hspace{1 in} or $[1,10][1,4][1,4][2,57][2,3][1,2][1,2]n.$\\
% Both 3.70

Case {\bf XVII}: $n \equiv 17$ (mod 19). We may use $[1,8][1,4][2,57][2,3][1,2]n$.\\
% NOTE: DO NOT NEED [1,8] here. With [1,8] is 3.6556
\end{proof}

Proof of Lemma \ref{v_17=0}
\begin{proof} 

We may assume that
$v_2 \geq 10$, $v_3 \geq 4$, $v_5 \geq 2$, $v_7 \geq 1$ $v_{11} \geq 1$ $v_{13} \geq 1$ $v_{19}\geq 1$.

Case {\bf I}: $n \equiv 1$ (mod 17). We may use $[1,34]n$. \\

Case {\bf II}: $n \equiv 2$ (mod 17). We may use $[2,51]n$. \\

Case {\bf III}: $n \equiv 3$ (mod 17). We may use $[1,18][1,34][1,2]n$.  \\

Case {\bf IV}: $n \equiv 4$ (mod 17). We may use $[1,18][4,85]n$. \\

Case {\bf V}: $n \equiv 5$ (mod 17). We may use $[1,18][1,34][2,3]n$. \\

Case {\bf VI}: $n \equiv 6$ (mod 17). We may use $[1,18]^2[6,119]n$. \\

Case {\bf VII}: $n \equiv 7$ (mod 17). We may use $[1,18][1,34][1,2]^3n$. \\

Case {\bf VIII}: $n \equiv 8$ (mod 17). We may use $[1,18][2,51][2,3]^2n$. \\

Case {\bf IX}: $n \equiv 9$ (mod 17). We may use $[1,18][4,85][1,2]n$.  \\

Case {\bf X}: $n \equiv 10$ (mod 17). We may use $[1,18]^2[6,119][1,2]^2n$.  \\

Case {\bf XI}: $n \equiv 11$ (mod 17). We may use $[1,18][1,34][2,3][1,2]n$. \\

Case {\bf XII}: $n \equiv 12$ (mod 17). We may use  $[1,10][1,6][1,18][4,85][1,2][2,3]n$.  \\

% 3.7067 
Case {\bf XIII}: $n \equiv 13$ (mod 17). We may  use $[1,18]^2[6,119][1,2]n$. \\

Case {\bf XIV}: $n \equiv 14$ (mod 17). We may use $[1,6][1,18][4,85][2,3]n$. 
% Another close call. 3.7169 . Not quite as close but another where we need all four threes.
\\

Case {\bf XV}: $n \equiv 15$ (mod 17). We may use $[1,18]^2[1,34][1,2]^4n$. \\
\end{proof}

Acknowledgements: Harry Altman and Juan Arias de Reyna provided extensive feedback. Juan Arias de Reyna provided substantial and detailed suggestions, including making explicit Lemma \ref{L:iterated}, checking and supplying correct versions of some the final gadgets, strengthening and simplifying the final gadgets, along with many other suggestions both large and small which greatly improved the presentation of the results.

%Consider adding in references to https://arxiv.org/abs/1711.01722 and  
% https://arxiv.org/abs/1709.04143

\end{document}